\documentclass[11pt]{article}
\usepackage[utf8]{inputenc}
\usepackage[english]{babel}
\usepackage{amsmath}
\usepackage{amsfonts,bm}
\usepackage{amssymb}
\usepackage{graphicx}

\usepackage{amssymb,amsmath,amsthm,enumerate}
\usepackage{fullpage}

\usepackage{hyperref}

\usepackage{amsthm}

\newtheorem{lemma}{Lemma}
\newtheorem{remark}{Remark}
\newtheorem{proposition}[lemma]{Proposition}

\newtheorem{theorem}[lemma]{Theorem}
\newtheorem{corollary}[lemma]{Corollary}

\newcommand{\VAR}{\operatorname{Var}}
\newcommand{\VARs}{\VAR_s \hspace{-2pt}}
\newcommand{\COS}{\cos^+ \hspace{-1pt}}

\newcommand{\LIP}{\mathrm{Lip}_0}

\newcommand{\IND}{\mathbf{1}}
\newcommand{\sg}{\mathrm{sg}}
\newcommand{\ac}{\mathrm{ac}}
\newcommand{\KS}{\mathrm{KS}}

\newcommand{\cT}{\mathcal{T}}

\newcommand{\cE}{\mathcal{E}}
\newcommand{\cF}{\mathcal{F}}

\newcommand{\Poi}{\mathrm{Poi}}
\newcommand{\Bin}{\mathrm{Bin}}
\newcommand{\UGW}{\mathrm{UGW}}
\newcommand{\GW}{\mathrm{GW}}
\newcommand{\dE}{\mathbb{E}}
\newcommand{\dP}{\mathbb{P}}
\newcommand{\dN}{\mathbb{N}}
\newcommand{\dR}{\mathbb{R}}
\newcommand{\dC}{\mathbb{C}}
\newcommand{\dH}{\mathbb{H}}

\newcommand{\dEs}{\mathbb{E}_{s}}
\newcommand{\g}[1]{\gamma \left( #1 \right) }
\newcommand{\G}[1]{\gamma^{p} \hspace{-2pt}\left( #1 \right) }

\newcommand{\veps}{\varepsilon}

\title{Existence of absolutely continuous spectrum for Galton-Watson random trees} 

\author{Adam Arras \and Charles Bordenave}

\date{}

\begin{document}

\maketitle

\begin{abstract}
We establish a quantitative criterion for an operator defined on a Galton-Watson random tree for having an absolutely continuous spectrum. For the adjacency operator, this criterion requires that the offspring distribution has a relative variance below a threshold. As a by-product, we prove that  the adjacency operator of  a supercritical Poisson Galton-Watson tree has a non-trivial absolutely continuous part if the average degree is large enough. We also prove that its Karp and Sipser core has purely absolutely spectrum on an interval   if the average degree is large enough. We finally illustrate our criterion on the Anderson model on a $d$-regular infinite tree with $d\geq 3$ and give a quantitative version of Klein's Theorem on the existence of absolutely continuous spectrum at  disorder smaller that $C \sqrt d$ for some absolute constant $C$.
\end{abstract}

\section{Introduction}

In this paper, we study operators on an infinite random tree. We establish a quantitative criterion for the existence of an absolutely continuous spectrum which relies on the  variance of the offspring distribution.

\paragraph{Galton-Watson trees.}

The simplest ensemble of random trees are the Galton-Watson trees. Let $P$ be a probability distribution on the integers $\dN = \{0,1,\ldots\}$. We define $\GW(P)$ as the law of the tree $\cT$ with root $o$ obtained from a Galton-Watson branching process where the offspring distribution has law $P$ (the root has a number of offspring $N_o$ with distribution $P$ and, given $N_o$, each neighbour of the root, has an independent number of offspring with distribution $P$, and so on). For example, if $P$ is a Dirac mass at integer $d\geq 2$, then $\cT$ is an infinite $d$-ary tree.

An important current motivation for studying random rooted trees is that they appear as the natural limiting objects for the Benjamini-Schramm topology of many finite graphs sequences, see for example \cite{zbMATH06653786,coursSRG}. The limiting random rooted graphs that can be limiting points for this topology must satisfy a stationarity property called unimodularity. The law $\GW(P)$ is generically not unimodular. To obtain a unimodular random rooted tree, it suffices to biais the law of the root vertex. More precisely, if $P_\star$ is a probability distribution on $\dN$ with finite positive first moment $d_\star = \sum_{k} k P_\star(k)$, we define $\UGW(P_\star)$ as the law of the tree $\cT$ with root $o$ obtained from a Galton-Watson branching process where the root has offspring distribution $P_\star$ and all other vertices have offspring distribution $P = \widehat P_\star$ defined for all integer $k\geq 0$ by 
\begin{equation}\label{eq:defhatP}
P(k) = \frac{(k+1) P_\star (k+1)}{d_\star}.
\end{equation}
Such a random rooted tree is also called a {\em unimodular Galton-Watson tree with degree distribution $P_\star$} by distinction with the {\em Galton-Watson tree with offspring distribution $P$} defined above. For example, if $P_\star$ is a Dirac mass at $d_\star$ then $\cT$ is an infinite $d_\star$-regular tree. If $P_\star = \Poi(d)$ is a Poisson distribution with mean $d$, then $\widehat P_\star = P_\star$ and the laws $\UGW(P_\star)$ and $\GW(P_\star)$ coincide. This random rooted tree is the Benjamini-Schramm limit of the Erd\H{o}s-R\'enyi random graph with $n$ vertices and connection probability $d/n$ in the regime $d$ fixed and $n$ goes to infinity. More generally, uniform random graphs whose given degree sequence converge under mild assumptions to $\UGW(P_\star)$ where $P_\star$ is the limiting degree distribution, see  for example  \cite{zbMATH06653786}.

The {\em percolation problem} on such random trees is their extinction/survival probabilities, defined as the probability that the tree $\cT$ is finite/infinite. We recall that for $\GW(P)$, with $P$ different to a Dirac mass at $1$, the survival probability is positive if and only if $d = \sum_k k P(k) >1$. Obviously, for $\UGW(P_\star)$, the survival probability is positive if and only the survival probability of $\GW(P)$ is positive with $P = \widehat P_\star$ as in \eqref{eq:defhatP}.

\paragraph{Adjacency operator.} Let $\cT$ be an infinite rooted tree which is locally finite, that is for any vertex $u \in \cT$, the set of its neighbours denoted by $\mathcal{N}_u$ is finite. For $\psi \in \ell^2(\cT)$, with finite support, we set
$$(A\psi)(u) =  \sum_{v\in \mathcal{N}_u} \psi(v), \quad u \in \cT.$$
This defines an operator $A$ on a dense domain of $\ell^2(\mathcal T)$.  Under mild conditions on $\cT$, the operator $A$ is essentially self-adjoint, see for example \cite[Proposition 3]{MR2789584}.  For example, if $\cT$ has law $\GW(P)$ or $\UGW(P)$ with $P$ having a finite first moment, then $A$ is essentially self-adjoint with probability one, see   \cite[Proposition 7]{MR2789584} and \cite[Proposition 2.2]{coursSRG} respectively. We may then consider the spectral decomposition of $A$ and define, for any $v \in \cT$,
$\mu_v  $ the {\em spectral measure at vector $\delta_v$} which is a probability measure on $\dR$ characterized for example through its Cauchy-Stieltjes transform: for all $z \in \dH= \{ z \in \dC : \Im(z) >0\}$, 
$$
\int_{\dR} \frac{\mathrm{d} \mu_v(\lambda)}{\lambda - z} = \langle \delta_v , (A - z) ^{-1}  \delta_v \rangle.
$$
If $\cT$ is random and $o$ is the root of the tree, the {\em average spectral measure or density of states} is the deterministic probability measure $\bar \mu = \dE \mu_o$.  The probability measure $\bar \mu$ has been extensively studied notably when $\cT$ is a Galton-Watson random tree: the atoms are studied in \cite{MR2789584,SALEZ2015249,zbMATH07174134}, the existence of a continuous part in \cite{zbMATH06821385} and the existence of an absolutely continuous part at $E = 0$ in \cite{CS18}, see \cite{coursSRG} for an introduction on the topic. We mention that for $\GW(P)$ or $\UGW(P_\star)$,  $\bar \mu$ has a dense atomic part as soon as the tree has leaves (that is $P(0) > 0$).

\paragraph{Quantum percolation.}  Much less is known however on the random probability measure $\mu_o$. Its decomposition into absolutely continuous and singular parts
$$
\mu_o = \mu^{\ac}_o + \mu^{\sg}_o
$$
 is however of fundamental importance to understand the nature of the eigenwaves of $A$, notably in view of studying the eigenvectors of finite graph sequences which converge to $\cT$, see notably \cite{zbMATH06111055,zbMATH06176899,MR3411739,zbMATH07068259,zbMATH07097490}. The {\em quantum percolation problem} is to determine whether  $\mu_o^{\ac}$ is trivial or not, it finds its origin 
in the works of De Gennes, Lafore and Millot \cite{doi:10.1142/9789812564849_0001,DEGENNES1959105}. The quantum percolation is a refinement of the classical percolation in the sense that a necessary condition for the non-triviality of $\mu_o^{\ac}$ is that the tree $\cT$ is infinite. We refer to \cite{MS10,coursSRG} for further references on quantum percolation.

In the present paper, we establish a general concentration criterion on the offspring distribution $P$ of a Galton-Watson tree which guarantees the existence of an absolutely continuous part. Before stating our main results in the next section, we illustrate some of its consequences on $\UGW(\Poi(d))$ and $\UGW(\Bin(n,d/n))$ where for integer $n \geq 1$ and $0 \leq d \leq n$, $\Bin(n,d/n)$ is the Binomial distribution with parameters $n$ and $d/n$. The random tree is $\UGW(\Bin(n,d/n))$ is particularly interesting because this is the random tree obtain after deleting independently each edge of the $n$-regular infinite tree with probability $d/n$. In the same vein, $\UGW(\Poi(d))$ is the Benjamini-Schramm limit of the bond-percolation with parameter $d/n$ on the hypercube $\{0,1\}^n$. As pointed above, in both cases $\bar \mu$ has a dense atomic part, so even if $\mu_o^{\ac}$ is non-trivial, $\mu_o$ will contain an atomic part with positive probability.

\begin{theorem}\label{thm:POI}
Assume that $\cT$ has law $\UGW(\Poi(d))$ or $\UGW(\Bin(n,d/n))$ with $1 \leq d \leq n$. For any $\veps >0$, there exists $d_0= d_0 (\veps) >1 $ such that if $ d \geq d_0$ then, conditioned on non-extinction,  with probability one, $\mu^{\mathrm{ac}}_o$ is non-trivial and  $ \dE [ \mu^{\mathrm{ac}}_o(\dR) ] \geq  1-  \veps$.
\end{theorem}

In the proof of Theorem \ref{thm:POI}, we will exhibit a deterministic Borel set $B$ of Lebesgue measure proportional to $\sqrt d$ such that, conditioned on non-extinction, with probability one, $\mu^{\mathrm{ac}}_o$ has positive density almost everywhere on $B$.  This theorem is consistent with the prediction in Harris \cite{PhysRevB.29.2519} and Evangelou and Economou \cite{PhysRevLett.68.361}. It should be compared to \cite{MR3411739} which proved a similar statement for $\UGW(\Bin(n,p))$ with $n$ fixed and $p$ close to $1$, that is a random tree which is close to the $n$-regular infinite tree. To our knowledge, the present result establishes for the first time the presence of a non-trivial absolutely continuous part in a random tree which is not close to a deterministic tree. We note also that Theorem \ref{thm:POI} implies as a corollary that the average spectral measure $\bar \mu = \dE \mu_o$ has an absolutely continuous part of arbitrary large mass if $d$ is large enough, this result was not previously known.

As pointed above, if $\cT$ has law $\UGW(\Poi(d))$ or $\UGW(\Bin(n,d/n))$ then, due to leaves in $\cT$, $\mu^{\mathrm{ac}}_o$ has atoms.  It was however suggested in Bauer and Golinelli \cite{BG01a,BG01} that the {\em Karp and Sipser core} of $\cT$ should carry the absolutely continuous part of $\mu_o$. This core is the random subforest of $\cT$ obtained by iteratively repeating the following procedure introduced in \cite{KS}: pick a leaf of $\cT$ and remove it together with its unique neighbour, see \cite{MR2789584} for a formal definition. The core, if any, is the infinite connected component that remains. For our purposes, it suffices to recall that for $d > d_{\KS}$, the core is non-empty with probability one and empty otherwise (for $\Poi(d)$, we have $d_{\KS} = e$). Moreover, for $d > d_{\KS}$, conditioned on the root being in the core, the connected component of the root in the core has law $\UGW(Q)$ where for $\Poi(d)$, $Q$ is $\Poi(m(d))$ conditioned on being at least $2$ while for $\Bin(n,d/n)$, $Q$ is $\Bin(n,m_n(d)/n)$ conditioned on being at least $2$. We have for $d > d_{\KS}$, $m(d), m_n(d) > 0$ and $m(d)\sim m_n(d) \sim d$ as $d$ grows, see  \cite{MR1637403,Zdeborov__2006} or Lemma 16 in v1 arxiv version of \cite{MR2789584}.

The following theorem supports the predictions of Bauer and Golinelli.

\begin{theorem}\label{thm:POIcore}
Assume that $\cT$ has law $\UGW(Q_d)$ where $Q_d$ is either $\Poi(d)$ or $\Bin(n,d/n)$ conditioned on being at least $2$ with $0 < d \leq n$ and $n \geq 2$. The conclusion of Theorem \ref{thm:POI} holds for $Q_d$. Moreover, for any $0 < E < 2$, there exists $d_1 = d_1(E) >0$ such that if $d \geq d_1$ then, with probability one, on the interval $(-E \sqrt d , E \sqrt d)$, $\mu_o$ is absolutely continuous with  almost-everywhere positive density. 
\end{theorem}

This theorem should be compared with Keller \cite{MR2994759} which proves among other results a similar statement for $\Bin(n,p)$ conditioned on being at least $2$ with $n$ fixed and $p$ close to $1$, see also Remark (c) there. Theorem \ref{thm:POIcore} proves  that the core of $\UGW(\Poi(d))$ has purely absolutely continuous spectrum on a spectral interval of size proportional to $\sqrt d$ provided that $d$ is large enough. We expect that this result is close to optimal: almost surely, the spectrum of $\UGW(\Poi(d))$ is $\dR$ and the upper part of the spectrum should be purely atomic due to the presence of vertices of arbitrarily  
large degrees.

Let us also mention that for $\Poi(d)$ we expect that  $d_1(0+) = 0$ in accordance to \cite{BG01a,BG01, CS18}. For the optimal value of $d_0(0+)$ in Theorem \ref{thm:POI}  for $\Poi(d)$, $d_0(0+) = e\approx 2.72$ seems plausible in view \cite{BG01a,BG01, CS18} but older references \cite{PhysRevB.29.2519,PhysRevLett.68.361} predict that $d_0(0+) \approx 1.42$.

\paragraph{Anderson model.} The focus of this paper is on random trees, we nevertheless conclude this introduction with an application of our main result to the Anderson model on the infinite $d$-regular tree, say $\cT_d$. Let $A$ be the adjacency operator of $\cT_d$. Let $(V_u)_{u \in \cT_d}$ be independent and identically distributed random variables. We consider the essentially self-adjoint operator on $\ell^2(\cT_d)$,
$$
H =   A + \lambda V,
$$
where $\lambda \geq 0$ and $V$ is the multiplication operator $(V\psi) (u) = V_u \psi(u)$ for all $u \in \cT_d$. 
This famous model was introduced by Anderson \cite{PhysRev.109.1492} to describe the motion of
an electron in a crystal with impurities. The study of such operators on the infinite regular tree was initiated in the influential works \cite{ACTA,ACT}. We refer to the monograph \cite{zbMATH06532891} for further background and references.

As above, we denote  by $\mu_o$ the spectral measure at the vector $\delta_o$, where $o$ is the root of $\cT_d$. We prove the following statement: 

\begin{theorem}\label{thm:Anderson}
Assume that $V_o$ has finite fourth moment. For any $\veps >0$ there exists $\lambda_0  = \lambda_0(\veps) > 0$ such that for all $0 \leq \lambda \leq \lambda_0 \sqrt d$ and $d\geq 3$, $\dE [ \mu^{\ac}_o (\dR) ]\geq 1 -\veps$. Moreover, for any $0 < E < 2$, there exists $\lambda_1 = \lambda_1(E) >0$  such that for all $0 \le \lambda \leq \lambda_1 \sqrt d$  and $d\geq 3$, with probability one, on the interval $(-E \sqrt{ d-1} , E \sqrt{d-1})$, $\mu_o$ is absolutely continuous with  almost-everywhere positive density. 
\end{theorem}

For fixed $d \geq 3$, this theorem is contained in Klein \cite{klein1998extended}. Other proofs of this important result have been proposed \cite{zbMATH05064642,froese2007absolutely}. Beyond the proof, the novelty of Theorem \ref{thm:Anderson} is the uniformity of the thresholds $\lambda_0,\lambda_1$ in $d \geq 3$. When $V_o$ has a continuous density, the asymptotic for large $d$ was settled in Bapst \cite{zbMATH06358474} using the complementary criteria for pure point / absolutely continuous spectra established in \cite{zbMATH00487169,zbMATH06180736}.

\paragraph{Acknowledgments} We thank Nalini Anantharaman and Mostafa Sabri for their explanations on their work \cite{zbMATH07097490}.  CB is supported by French ANR grant ANR-16-CE40-0024-01.

\section{Main results}
\label{sec:main}

\paragraph{Model definition.} Let $P_\star$ be a probability distribution on $\dN$ with finite and positive second moment. We denote by $P$ the size-biased version of $P_\star$ defined by \eqref{eq:defhatP}. We denote by $N_\star$ and $N$ random variables with laws $P_\star$ and $P$ respectively. We set 
$$
d_\star = \dE N_\star \quad \hbox{ and } \quad d = \dE N = \frac{\dE N_\star (N_\star - 1)}{\dE N_\star}.
$$
We use Neveu notation for indexing the vertices of a rooted tree. Let $\mathcal V$ be the set of finite sequence of positive integers $u = (i_1,\ldots,i_k)$, the empty sequence being denoted by $o$. We denote by $\mathcal V_k$ the sequences of length $k$. To each element of $\mathcal V$, we associate an independent random variable $N_u$ where $N_o$ has law $P_\star$ and $N_u$, $u \ne o$, has law $P$. We define the offsprings of $u \in \mathcal V_k$ as being the vertices $(u,1),\ldots, (u,N_u) \in \mathcal V_{k+1}$ (by convention $(o,i) = i$). The connecting component of the root $o$ in this genealogy defines a rooted tree $\cT$. Its law is $\UGW(P_\star)$. More generally, for $u \in \mathcal V$, we denote by $\cT_u$ the trees spanned by the descendants of $u$ (so that $\cT_o = \cT$). If $u \ne o$, $\cT_u$ has law $\GW(P)$ (to be precise up to the natural isomorphism which maps $u$ to $o$). By construction, the degree of a vertex $u \in \cT$ is $\deg(u) = \IND_{u \ne o} + N_u $.

Let $(X_u)_{u \in \mathcal V}$ be independent and uniform random variables in some measured space. Then, we consider real random variables $V_u = f(X_u,N_u,\IND_{u \ne o})$ for some measurable function $f$. Now, we introduce the operator defined on compactly supported vectors of $\ell^2 (\cT)$,
\begin{equation}\label{eq:defH}
H = \frac{A}{\sqrt{d}} - V,
\end{equation}
 where $A$ is the adjacency operator and $V$ the multiplication operator $V\psi(u) = V_u \psi(u)$.  As pointed above, from \cite[Proposition 7]{MR2789584}, with probability one, the operators $A$ and $H$ are essentially self-adjoint on a dense domain of $\ell^2(\cT)$.

 For example, if for all $u \in \mathcal V$, $V_u = 0$ then $H$ is the adjacency matrix normalized by $1/\sqrt d$. If $V_u = (\deg(u)-d) / \sqrt d $ then $- H$ is the Laplacian operator on $\cT$ shifted by $d$ and normalized. If $V_u = f(X_u)$ for some measurable function $f$ then we obtain the Anderson Hamiltonian operator on $\cT$. These three examples are relevant for our study.

\paragraph{Resolvent recursive equation.}  We denote by $\mu_o$ the spectral measure at the root vector $\delta_o$.  For  $z\in \dH = \{z \in \dC :  \Im z > 0\}$,  the resolvent operator $ (H-z)^{-1}$ is well-defined. The resolvent is the favorite tool to access the spectral measures on tree-like structures. We set
\begin{equation}\label{eq:defgo}
g_o(z)= \langle \delta_o , (H-z)^{-1} \delta_o \rangle = \int_{\dR} \frac{\mathrm{d}\mu_o(\lambda)}{\lambda-z}.\end{equation}
For $u \in \cT$, we denote by $H_u$ the restriction of $H$ to $\ell^2 (\cT_u)$ (in particular, $H_o = H$). If $g_u(z)= \langle \delta_u, (H_u - z)^{-1} \delta_u \rangle$, then we have the following recursion:
\begin{equation}\label{recursion}
    g_u(z) = - \left(z+\frac{1}{d}\sum_{i=1}^{N_u} g_{ui}(z)+V_u\right)^{-1},
\end{equation}
see for example \cite[Proposition 16.1]{zbMATH06532891}.
By construction, all $g_u$ with $u \ne o$, have the same law, let $g(z)$ be a random variable with law $g_u(z)$, $u \ne o$. For any $ z\in \dH$, we deduce the equations in distribution:
\begin{equation}
\label{RDEoz}g_o (z)\stackrel{d}{=} - \left( z + \frac 1 d \sum_{i=1}^{N_\star} g_i(z)  + v_\star\right)^{-1},
\end{equation}
and 
\begin{equation}\label{eq:RDEz}
g (z) \stackrel{d}{=}  - \left( z + \frac 1 d \sum_{i=1}^{N} g_i (z) + v\right)^{-1},
\end{equation}
where $(N_\star,v_\star)$ has law $(N_o,V_o)$, $(N,v)$ has law $(N_u,V_u)$, $u \ne o$, and $(g_i)_{i \geq 1}$ are independent iid copies of $g$.

The {\em semicircle distribution} is defined as $$\mu_{sc} = \frac{1}{2\pi} \mathbf{1}_{\{|x|\leq 2\}}\sqrt{4-x^2}\mathrm{d}x.$$ Its Cauchy-Stieltjes transform, $\Gamma(z)$, satisfies, for all $z \in \dH$,
\begin{equation}\label{screcursion}
  \Gamma(z)= -\left(z+\Gamma(z)\right)^{-1}.
\end{equation}
Similarly, we define the {\em modified Kesten-McKay distribution} with ratio $\rho = d_*/d$ as  
$$
\mu_{\star} =  \frac{2 \rho}{\pi} \mathbf{1}_{\{|x|\leq 2\}}\frac{\sqrt{4-x^2}}{ 2\rho^2 ( 4 - x^2) +  x^2 ( 2 - \rho)^2 }\mathrm{d}x,
$$
(if $d_* = d+1$, we retrieve the usual Kesten-McKay distribution). Its Cauchy-Stieltjes transform is related to $\Gamma$ through the identity for all $z \in \dH$,
\begin{equation}\label{screcursiono}
\Gamma_{\hspace{-2pt}\star} (z) = -\left(z+ \rho  \Gamma(z)\right)^{-1}.
\end{equation}
In the statements below, we give a probabilistic upper bound on $|g(z)-\Gamma(z)|$ and $|g_o(z)-\Gamma_\star(z)|$ uniformly in $z$ in the strip $$\dH_E = \{ z \in \dH : \Re z \in  (-E,E),\, \Im z \leq 1 \},$$ for $0 < E < 2$ provided that $v$ is small and the ratios $N/d$ and $N_\star/d_\star$ are close enough to $1$ in probability.

\paragraph{Trees with no leaves.} For a given $z \in \dH$, let $g(z)$ be a random variable in $\dH$ satisfying the equation in distribution \eqref{eq:RDEz}.  We allow $v$ to take value in $\dH$ for a reason which will be explained below. Then, given a real number $p \geq 2$ and a parameter $\lambda \in (0,1)$, we denote by $\mathcal{E}_\lambda$ the event
\begin{equation}\label{eq:defEl}
\mathcal{E}_\lambda = \left\{|v| < \lambda,\, N \geq 2,\, \frac{N}{d}\in (1-\lambda, 1+\lambda) \right\},
\end{equation}
and $\alpha_p(\lambda)$ the control parameter
\begin{equation}\label{eq:defalphal}
\alpha_p(\lambda)= \dE \left[\IND_{\cE_\lambda^c} \left(1+|v|\right)^{2p}\left(\frac{N}{d}+\frac{d}{N}\right)^{p}\right].
\end{equation}
Note that $\alpha_p(\lambda)$ is infinite if $\dE N^{p} = \infty$ or $\dE |v|^{2p} = \infty$. Also, H\"older inequality implies for example 
\begin{equation}\label{eq:holder}
\alpha_2(\lambda) \leq \left( \dP ( \cE_\lambda^c ) \right)^{1/2} \left( \dE \left[ \left(1+|v|\right)^{12} \right] \right)^{1/3}  \left( \dE \left[ \left(\frac{N}{d}+\frac{d}{N}\right)^{12} \right] \right)^{1/6}.
\end{equation}
The following result asserts that if $\lambda + \alpha_p(\lambda)$ is small enough and $N \geq 1$ with probability one then $g(z)$ and $\Gamma(z)$ are close in $L^{p}$.

\begin{theorem}\label{thm:main} Let $0 <E <2$, $p\geq2$ and $(N,v)\in \dN \times \dH$ be pair of random variables with $\dE N = d$. Assume $\dP ( N \geq 1) = 1$ and either $v$ real or $\dE N^{2p} < \infty$. There exist $\delta, C >0$ depending only on $(E,p)$ such that for any $g(z) \in \dH$ verifying \eqref{eq:RDEz}, if $\lambda + \alpha_p(\lambda) \leq \delta$ 
then for all $z \in \dH_E$,
$$\dE \left[ |g(z)-\Gamma(z)|^{p} \right] \leq C\sqrt{\lambda + \alpha_{p}(\lambda)},$$
$$\dE \left[ (\Im g(z))^{-p} \right]\leq C .$$
\end{theorem}

The strategy of proof is inspired by the geometric approach to Klein's Theorem initiated by Froese, Hasler and Spitzer in \cite{froese2007absolutely} and refined in Keller, Lenz and Warzel \cite{keller2012absolutely}. The main technical novelty is that Theorem \ref{thm:main} holds uniformly on the distribution of $(N,v)$. To achieve this, we use some convexity and symmetries hidden in \eqref{eq:RDEz}, as a result, our proof is at the same time simpler than \cite{froese2007absolutely,keller2012absolutely} and more efficient.

We will obtain the following corollary of Theorem \ref{thm:main} which applies notably to  operators on random trees with law $\UGW(P_\star)$ when the minimal degree is $2$. Let $\mathcal{F}_\lambda$ and $ \beta_p(\lambda)$ be 
$$\mathcal{F}_\lambda = \left\{\frac{|v_\star|}{\rho} < \lambda,\, \frac{N_\star}{d_\star} \in (1-\lambda, 1+\lambda) \right\},$$
$$
\beta_p(\lambda) = \dE \left[\IND_{\cF_\lambda^{ c}} \left(1+\left|\frac{v_\star}{\rho}\right|\right)^{2p}\left( \frac{N_\star}{d_\star} + \frac{d_\star}{N_\star} \right)^{p}\right],$$
where $\rho = d_\star/d$ and recall the definition of $\Gamma_\star$ in \eqref{screcursiono}.

\begin{corollary}\label{cor:main} Let $0 <E <2$, $p\geq2$ and $(N_\star,v_\star)$, $(N,v)$ be random variables in $ \dN \times \dH$ with $\dE N_\star = d_\star$, and $\dE N = d$. Assume $\dP ( N_\star  \geq 1)  =\dP ( N \geq 1) = 1$ and either $v$ real or $\dE N^{2p} < \infty$. Let $\delta >0$ be as in Theorem \ref{thm:main}. There exists $C_\star >0$ depending only on $(E,p)$ such that for any  $g_o(z)$ and $g(z)$ satisfying \eqref{RDEoz}-\eqref{eq:RDEz}, if $\lambda + \alpha_p (\lambda)  \leq \delta$ then for all $z \in \dH_E$,
$$\dE \left[ |g_o(z)-\Gamma_{\hspace{-2pt}\star}(z)|^{p} \right] \leq 
C_\star \rho^{-p}  \max\left\{ \veps(\lambda) , \sqrt{\veps(\lambda)}\right\},$$
$$\dE \left[ (\Im g_o(z))^{-p}\right] \leq C_\star \left( 1 + \beta_p(\lambda) \right)^p \left(\rho + \frac{1}{\rho}\right)^{p},$$
where $\veps(\lambda) = \big(1+\lambda + \alpha_p(\lambda)\big)\big(1+\lambda + \beta_p(\lambda)\big)-1$.
\end{corollary}

We postpone to Section \ref{sec:specmeas} some well-known consequences of this corollary on the spectral measure $\mu_o$ of $H$.

\paragraph{Trees with leaves.} If $\dP ( N = 0) >0$, the random tree $\cT$ with law $\UGW(P_\star)$ has leaves. For the adjacency operator, even when the tree is infinite, this creates atoms in the spectrum located at all eigenvalues of the finite pending trees which have positive probability, see \cite{coursSRG}. Nevertheless a strategy introduced in \cite{MR3411739} allows to use results such as Theorem \ref{thm:main} to prove existence of an absolutely continuous part.

For simplicity, we restrict ourselves to the adjacency operator $A$. As already pointed if $d = \dE N > 1$, then $\cT$ is infinite with positive probability. The probability of extinction $\pi_e$ is the smallest solution in $[0,1]$ of the equation
$
x = \varphi(x), 
$
where $\varphi$ is the moment generating function of $P$:
$$
\varphi(x ) = \dE \left[ x^{N}\right]= \sum_{k=0}^{\infty} P_k x^k .
$$
Note that if $\varphi_\star$ is the moment generating of $P_\star$ then 
$
\varphi(x) = \varphi_\star'(x) / \varphi_\star'(1).
$

Next, we define the subtree $\mathcal S \subset \mathcal T$ of vertices $v \in \cT$ such that $\cT_v$ is an infinite tree. If $\cT$ is infinite (that is, $o \in \mathcal S$), this set $\mathcal S$ is the infinite skeleton of $\cT$ where finite pending trees are attached.  If $v \in \mathcal V$,  let $N^s_{v}$ be the number of offsprings which are in $\mathcal S$. By construction, if $v \in \mathcal S$ then $N^s_v \geq 1$. Conditioned on $\cT$ or $\cT_v$, $v\ne o$, is infinite, the moment generating functions of $N^s_o$ and $N^s_v$ are 
\begin{align}\label{eq:defMGFs}
\varphi_\star^s (x)  = \frac{\varphi_\star( (1 - \pi_e) x + \pi_e ) -\varphi_\star(\pi_e)  }{1 - \varphi_\star(\pi_e) }\quad \hbox{ and } \quad 
 \varphi^s (x) &= \frac{\varphi( (1 - \pi_e) x + \pi_e ) - \pi_e }{1 - \pi_e},
\end{align}
see Subsection \ref{subsec:skeleton} for more details. We denote by $N^s_\star$ and $N^s$ random variables with moment generating functions $\varphi_\star^s$ and $\varphi^s$. We set $d_s = \dE N^s = \varphi'(1-\pi_e)$.

We then consider the rescaled adjacency operator
$$
H = \frac{A}{\sqrt d_s},
$$ 
and let $g_o(z) = \langle \delta_o , (H-z)^{-1} \delta_o \rangle $ be as in \eqref{eq:defgo}. By decomposing the recursion \eqref{RDEoz}-\eqref{eq:RDEz} over offspring in $\mathcal S$ and in $\mathcal T \backslash \mathcal S$, we will obtain the following Lemma which makes the connection between $g_o(z)$ and a recursion on $\mathcal S$ of the type  studied above.

\begin{lemma}\label{lem:RDEleaves} 
Assume $d > 1$. For each $z \in \dH$, there exist random variables $v_\star(z)$ and $v(z)$ in $\dH$, such that, conditioned on $\cT$ is infinite,
\begin{equation*}
g_o (z)\stackrel{d}{=} - \left( z + \frac 1 {d_s} \sum_{i=1}^{N^s_\star} g^s_i(z)  + v_\star(z)\right)^{-1},
\end{equation*}
and 
\begin{equation*}
g^s (z) \stackrel{d}{=}  - \left( z + \frac 1 {d_s} \sum_{i=1}^{N^s} g^s_i (z) + v(z)\right)^{-1},
\end{equation*}
where $(g^s_i)_{i \geq 1}$ are independent iid copies of $g^s$, independent of $(N^s_\star,v_\star(z))$, $(N^s,v(z))$.
\end{lemma}

Since $N_s \geq 1$ with probability one, it is possible to use Theorem \ref{thm:main} provided that we have a moment estimate on $|v(z)|$ and $|v_\star(z)|$, see \eqref{eq:holder}. This is the content of the following technical lemma. It is a refined quantified version of \cite[Proposition 22]{MR3411739}. If $B\subset \dR$ is a Borel set, we define 
\begin{equation}\label{eq:defHB}
\dH_B = \{ z \in \dH : \Re(z) \in B,\,  \Im z \leq 1\}.
\end{equation}
With our previous notation for $E >0$, $\dH_E = \dH_{(-E,E)}$. We also denote by $\ell( \cdot)$ the Lebesgue measure on $\dR$.
Recall that $N$ is the random variable with law $P$ and $d = \dE N$.

\begin{lemma}\label{lem:leaves} 
Let $\veps >0$. There exist universal constants $c_0,C >0$ and a deterministic Borel set $B = B(\veps) \subset \dR$ with $\ell (B^c) \leq \veps$  such that the following holds.  Let $\pi_1 = \dP ( N \leq 1)$. If $4 \pi_1 \dE N^2 \leq d$ and $m \geq 1$ is an integer such that $\pi_1 \leq c_0^m$ then for any $ z\in \dH_B  $,
$$
\dE  \left[ |v(z)|^m \right]  \leq \left( \frac{C }{  \veps  d  } \right)^m  \frac{\dE [ N^m ]}{d^{m}}  d \pi_1, 
$$
and 
$$
\dE  \left[ |v_\star(z)|^m \right]  \leq \left( \frac{C }{  \veps  d  } \right)^m  \frac{\dE [ N_\star^m ]}{d^{m}}  d \pi_1. 
$$
\end{lemma}

Lemma \ref{lem:leaves} implies that $v(z)$ and $v_\star(z)$ are small if $N/d$ is concentrated around $1$.  As a byproduct of Lemma \ref{lem:RDEleaves} and Lemma \ref{lem:leaves}, we can use Theorem \ref{thm:main} and Corollary \ref{cor:main} to obtain an estimate on $|g^s(z) - \Gamma(z)|$ and $|g_o(z) - \Gamma_\star(z)|$ conditioned on non-extinction, for $z \in \dH_E \cap \dH_B$. We will see an application in the proof of Theorem \ref{thm:POI}. We note that from the proof, the Borel set $B$ is fully explicit but not simple: $B^c$ is dense in $\dR$. We note also that the set $B$ depends on $\veps$ and $d_s$ (through a simple scaling of its complementary).

\paragraph{Discussion.} The proofs in this paper are quantitative. For example, given $(E,p)$, we could produce in principle numerical constants for the parameters $\delta$ and $C$ in Theorem \ref{thm:main}. Since these constants should not be optimal, in order to avoid  unnecessarily complications of the proofs, we have not fully optimized our arguments and the statements above.

There is an important application of our main result for the eigenvectors of finite graphs converging in Benjamini-Schramm sense to $\UGW(P_\star)$. In \cite{zbMATH07097490}, the authors establish a form of quantum ergodicity for eigenvectors of finite graphs under some assumptions on the graph sequence. A key assumption called (Green) in \cite{zbMATH07097490} is implied by Theorem \ref{thm:main} for $p$ large enough (more precisely, the conclusion $\dE (\Im g(z))^{-p} \leq C$ for all $z \in \dH_E$ gives a weak form of (Green) which is sufficient to apply the main result in \cite{zbMATH07097490}, see \cite[Subsection 3.2]{MR4044435}). For example, the assumptions of \cite[Theorem 1.1]{zbMATH07097490} are met with probability one for a sequence of random graphs sampled uniformly among all graphs with a given degree sequence with degrees between $3$ and a uniform bound $D$ and whose empirical degree distribution converges to $P_\star$ for which we can apply the conclusion of Corollary \ref{cor:main} with $p$ large enough. This answers Problem 4.4 and Problem 4.5 in \cite{MR4044435} for the weak form of (Green).

 \paragraph{Organization of the rest of the paper.} In Section \ref{sec:pfmain}, we prove Theorem \ref{thm:main} and Corollary \ref{cor:main}. In Section \ref{sec:leaves}, we prove Lemma \ref{lem:RDEleaves}  and Lemma \ref{lem:leaves}. In the final Section \ref{sec:specmeas}, we gather some standard properties relating the spectral measures of an operator and its resolvent. As a byproduct, we prove Theorem \ref{thm:POI}, Theorem \ref{thm:POIcore} and Theorem \ref{thm:Anderson} from the introduction.

\section{Proof of Theorem \ref{thm:main} and Corollary \ref{cor:main}}
\label{sec:pfmain}
\subsection{Hyperbolic semi-metric}

\paragraph{Froese-Hasler-Spitzer's strategy.} Following \cite{froese2007absolutely}, we see $\dH$ as the Poincaré half-plane model endowed with the usual hyperbolic distance $\mathrm{d}_{\dH}$. The key observation is that the application $h \mapsto \hat{h}=-(z+h)^{-1}$, seen in $(\dH,\mathrm{d}_{\dH})$ acts as a contraction and thus determines $\Gamma(z)$ as its unique fixed point. As in \cite{froese2007absolutely}, we introduce the functional $\g{\cdot,\cdot}$ defined by  
$$\g{g,h} = \frac{|g-h|^2}{\Im g \Im h }, \qquad g,h \in \dH,$$ 
which is related to the usual hyperbolic distance by $\mathrm{d}_{\dH}(g,h) = \cosh^{-1}(1+\g{g,h}/2)$. The next lemma allows us to derive usual bounds from the ones obtained in the hyperbolic plane.
\begin{lemma}[From hyperbolic to Euclidean bounds]\label{lem:boundgamma}
\,
For any $g,h \in \dH$, we have
\begin{equation*}
    |g - h|^2 \leq |h|^2\left(4\gamma^2(g,h) + 2\gamma(g,h)\right) \quad \hbox{ and } \quad \frac{1}{\Im g} \leq \frac{4\gamma(g,h)+2}{\Im h}.
\end{equation*}

\end{lemma}
\begin{proof} We start with the second inequality. We may assume that $\Im g \leq \Im h /2$, the bound is trivial otherwise. In particular we have $|g-h| \geq \Im h /2$, thus
$$\frac{1}{\Im g} \leq \frac{4|g-h|^2}{\Im g \Im h^2} = \frac{4\gamma(g,h)}{\Im h}.$$
We obtain the claim bound. For the first inequality, considering the case $|g| \leq 2|h|$ and $|g| \geq 2 |h|$, we find similarly for any $g,h \in \dH$
$$|g| \leq 4\g{g,h}\Im h + 2|h|,$$
see  \cite{froese2007absolutely,keller2012absolutely}. 
From $|g - h|^2  = \gamma(g,h) \Im g  \Im h$ and $\Im g \leq |g|$, we get
$$|g - h|^2 \leq 4\gamma^2(g,h) (\Im h)^2 + 2\gamma(g,h) |h| \Im h,$$
this implies the first bound.   \end{proof}

For ease of notation, we write 
$$\g{h} = \g{h,\,\Gamma(z)}, \quad h \in \dH,$$
 where the dependency in $z$ is implicit.
 
Our strategy is the  following. Let $g = g(z)$ satisfying \eqref{eq:RDEz}. For some transformation $\phi \colon \dH \to \dH$ verifying the equation in distribution $g \stackrel{d}{=} \phi(g) $, we want to obtain uniformly in $z$ a contraction of the following type
\begin{equation}\label{eq:contraction}
  \forall h\in \dH : \G{\phi(h)} \leq \left(1-\veps \mathbf{1}_{\{h \not\in K\}}\right) \G{h}   ,
\end{equation}
where $K$ is an hyperbolic compact of $\dH$. Denoting $R = \max_{h\in K} \G{h}$ and assuming that $\dE \G{g}$ is finite, this implies that $\dE \G{g} \leq R/\veps$. We will derive the desired bound on $\dE |g-\Gamma|^{p}$ and $\dE (\Im g)^{-p}$ from Lemma \ref{lem:boundgamma}.

\paragraph{Properties of $\gamma$.} The function $h \mapsto \g{h}$ has the following properties:
\begin{lemma}[Convexity $\gamma$]\label{prop:propgamma}
\,
\begin{itemize}
\item[(i)]  For any $h \in \dH,\, \g{-(z+h)^{-1}} \leq \g{h}$.
    \item[$(ii)$] The function $\gamma$ is a convex function.
    \item [$(iii)$] For any integer $n\geq1$ and $(h_1,\ldots h_n) \in \dH^n$:
    $$\g{\frac{1}{n}\sum_{i=1}^n h_i} = \frac{1}{n} \sum_{i,j=1}^n \cos \alpha_{ij}\sqrt{q_i \g{h_i}}\sqrt{q_j \g{h_j}},$$
    where $q_i = \frac{\Im h_i }{\sum_{j} \Im h_j}$ and $\alpha_{ij}=\arg (h_i -\Gamma)  \overline{(h_j - \Gamma)}$ with the convention that $\arg 0 = \pi/2$.
\end{itemize}
\end{lemma}
\begin{proof}
Statement $(i)$ comes from the fact that $h \mapsto \hat{h}=-(z+h)^{-1}$ is a contraction for the hyperbolic distance  $\mathrm{d}_{\dH}$ together with the fact that $\hat{\Gamma}=\Gamma$. Statement $(ii)$ follows from $(iii)$, Cauchy–Schwarz inequality and $\sum_i q_i = 1$,
$$\g{\frac{1}{n}\sum_{i=1}^n h_i} \leq \frac{1}{n} \sum_{i,j} \sqrt{q_i \gamma_i}\sqrt{q_j \gamma_j} \leq  \frac{1}{n}\left( \sum_{i} \gamma_i \right) \left( \sum_i q_i \right) = \frac{1}{n}\sum_i \g{h_i}.$$

The last statement $(iii)$ follows by expanding the squared term in the definition of $\gamma$:
\begin{align*}
\gamma\left(\frac{1}{n}\sum_{i} h_i\right) & = \frac{|\frac{1}{n}\sum_{i} h_i-\Gamma|^2}{\frac{1}{n} \sum_{i} \Im h_i \Im \Gamma} = \frac{1}{n \Im \Gamma} \frac{1}{\sum_j \Im h_j} \sum_{i,j} |h_i-\Gamma||h_j-\Gamma| \cos \alpha_{ij} .
\end{align*}
We obtain the requested identity $(iii)$.
\end{proof}

The next lemma gathers properties of $\gamma$. In the sequel, for $\lambda_0 > 0$, we denote by $\LIP(\lambda_0) $ the set of functions $f : [0,\lambda_0) \to \dR^+$ such that for some $C > 0$, $f(t) \leq Ct$ for all $t \in [0,\lambda_0/2]$.

\begin{lemma}[Regularity of $\gamma$]\label{prop:per1} 
Let  $0 < E < 2$, $p \geq 2$ and $z \in \dH_E$.
\begin{enumerate}
\item[(i)] 
There exists a function $c \in \LIP(1)$ depending only on $(E,p)$ such that for any $h \in \dH, \lambda \in [0,1)$, and $(u,v)\in \dR \times \dH$ with $|u|,|v| \leq \lambda$,
$$\gamma^q \big( (1+u) h +  v \big) \leq (1+c(\lambda))\gamma^q(h) + c(\lambda), \qquad q \in [1,p]$$
\item[(ii)] There exists $C = C(E,p)$ such that for any $h\in \dH,\mu > 0$ and $v\in \dH$,
$$\gamma^{p} \big( \mu h + v \big) \leq C (1+|v|^{2p})\left(\mu+ \frac{1}{\mu}\right)^{p}(\gamma^{p}(h) +1 ).$$
\end{enumerate}
\end{lemma}
\begin{proof}
For statement $(i)$ with $q=1$ we adapt \cite[Lemma 1]{keller2012absolutely}. There are two cases to consider. If $|h-\Gamma| > \Im \Gamma /2$, then
\begin{align*}
  \frac{\g{(1+u) h + v }}{\g{h}} &= \left(\frac{|(1+u) h + v -\Gamma|}{|h-\Gamma|}\right)^2 \frac{\Im h}{\Im ((1+u) h + v)}\\
  & \leq  \left(\frac{(1+\lambda)|h-\Gamma|+ \lambda ( 1 + |\Gamma|)}{|h-\Gamma|}\right)^2 \frac{1}{1-\lambda} \\
  &\leq \left(1+\lambda + \frac{\lambda(1 +|\Gamma|)}{\Im \Gamma/2} \right)^2 \frac{1}{1-\lambda}.
\end{align*}
Otherwise we have $|h-\Gamma| \leq \Im \Gamma /2$ and then $\Im h \geq \Im \Gamma/2$. We obtain similarly, using $2 ab \leq a^2 + b^2$ for real $a,b$ at the third line,
\begin{align*}
 \g{(1+u) h + v  } &\leq \frac{\left((1+\lambda)|h-\Gamma| + \lambda  ( 1 + | \Gamma|) \right)^2}{ ((1-\lambda)\Im h + \Im v )\Im \Gamma}\\
 & \leq \frac{(1+\lambda)^2|h-\Gamma|^2 + \lambda^2 (1+|\Gamma|)^2+2\lambda(1+\lambda)|h-\Gamma|(1+|\Gamma|)}{ (1-\lambda)\Im h\Im \Gamma} \\
  & \leq \frac{(1+\lambda)^2 + \lambda(1+\lambda)^2 }{1-\lambda}\frac{|h-\Gamma|^2}{\Im h \Im \Gamma}+\frac{\lambda^2 (1+|\Gamma|)^2+\lambda(1+|\Gamma|)^2}{ (1-\lambda)(\Im \Gamma)^2 /2} \\
  & = \frac{(1+\lambda)^3}{1-\lambda}\g{h}+\frac{\lambda(1+\lambda)(1+|\Gamma|)^2}{(1-\lambda)(\Im \Gamma)^2/2}.
\end{align*}
We have $\Gamma(z) = (- z + \sqrt{z^2 - 4} )/2$ where $\sqrt{\cdot}$ is the principal branch of the square root function (which maps $e^{i\theta}$ to $e^{i \theta/2}$ for $\theta \in [0,\pi]$). We deduce easily that $|\Gamma(z)| \leq 1$ and $\Im \Gamma(z)$ is lower bounded uniformly over $\dH_E$. Thus, for some function $c(\lambda)$ in $\LIP(1)$ we have 
\begin{equation}\label{eq:gammauhv}
\gamma \big( (1+u) h +  v \big) \leq (1+c(\lambda))\gamma(h) + c(\lambda).
\end{equation}
Let $q>1$, we claim that the general case follows up to a new function $\tilde{c}$. For any $x,c>0$, from a Taylor expansion, we have for some $0\leq \delta \leq c$,
$$((1+c)x+c)^q = (1+c)^qx^q+c(q-1)((1+c)x+\delta)^{q-1}.$$
Using that $a^{q-1} \leq a^q+1$, $(a+b)^{q} \leq 2^{q-1}(a^q+b^q)$ for $q\geq 1$, together with $0 \leq \delta \leq c$ we get 
$$((1+c)x+\delta)^{q-1} \leq 2^{q-1}(1+c)^qx^q+2^{q-1}c^q+1.$$
This leads to 
$$((1+c)x+c)^q \leq (1+c)^q(1+2^{q-1}(q-1)c)x^q+(2^{q-1}c^q+1)(q-1)c.$$
We get from \eqref{eq:gammauhv}
$$\gamma^q \big( (1+u) h +  v \big) \leq \left((1+c(\lambda))\gamma(h) + c(\lambda)\right)^q \leq (1+\tilde{c}(\lambda))\gamma(h)^q+\tilde{c}(\lambda),$$
where $\tilde{c}(\lambda)\in \LIP(1)$ is   
$$\tilde{c}= \max\left\{(1+c)^q(1+2^{q-1}(q-1)c) -1,(2^{q-1}c^q+1)(q-1)c \right\},$$
since we have that $c(\lambda)$ is in $\LIP(1)$.

For statement $(ii)$, we proceed similarly. If $|h-\Gamma| > \Im \Gamma /2$, then
\begin{align*}
\frac{\g{\mu h + v} }{\g{h}} & = \frac{\Im h}{\mu\Im h + \Im v}\left(\frac{|\mu (h-\Gamma) + v -(1-\mu)\Gamma|}{|h-\Gamma|}\right)^2 \\
& \leq \frac{1}{\mu}\left(\mu  + \frac{|v| + |1-\mu||\Gamma|}{\Im \Gamma /2}\right)^2 \\
& \leq 2\mu  + 16 \frac{|v|^2+|1-\mu|^2|\Gamma|^2}{\mu (\Im \Gamma)^2}.
\end{align*}
Where we used $(a+b)^2 \leq 2(a^2+b^2)$. Otherwise, we have $|h-\Gamma| \leq \Im \Gamma /2$ and $\Im h \geq \Im \Gamma/2$,
\begin{align*}
\gamma(\mu h + v) &\leq \frac{\left(\mu |h-\Gamma|+|v|+|1-\mu||\Gamma|\right)^2}{(\mu\Im h + \Im v)\Im \Gamma} \\
& \leq 2 \frac{\mu^2 |h-\Gamma|^2}{\mu \Im h \Im \Gamma} + 2\frac{(|v|+|1-\mu||\Gamma| )^2}{\mu (\Im \Gamma)^2 /2}\\
& \leq 2\mu \gamma(h) + 8 \frac{|v|^2+|1-\mu|^2|\Gamma|^2}{\mu (\Im \Gamma)^2}.
\end{align*}
We get for some constant $C$ depending on $E$,
$$\gamma \big( \mu h + v \big) \leq C (1+|v|)^{2}\left(\mu+ \frac{1}{\mu}\right)(\gamma(h) +1 ).$$
Using $(a+b)^{p}\leq 2^{p-1} (a^{p}+b^{p})$, this gives $(ii)$.
\end{proof}

\subsection{Asymmetric contraction}

In this subsection we derive our main technical tool. Although we noticed that $h \mapsto - (z+h)^{-1}$ is a contraction in $\dH$, it tends to an isometry as $z$ approaches the real axis and thus $\g{(z+h)^{-1}} \simeq \g{h}$. This prevent us to obtain contraction (\ref{eq:contraction}) for $\phi(h)=\hat{h}$, and this is precisely why the stability of absolutely continuous spectrum fails in one dimensional graphs such as $\mathbb{Z}$. We overcome this issue with the following observation. If instead we consider two points $h_r,h_l\in \dH$, then for some $\veps=\veps(E,p)> 0$ we show that
$$ \G{\frac{h_r + h_l}{2}} +\G{\frac{\hat{h}_r+h_l}{2}} \leq (1-\veps)\left( \G{h_r} + \G{h_l} \right),$$ 
as soon as $h_r$ and $h_l$ are far enough from $\Gamma$ for $\mathrm{d}_{\dH}$. The terms asymmetric is motivated from the fact that $h_l$ is fixed while the terms with subscripts $r$ varies on the left-hand side. 

More precisely, let $\lambda \in [0,1)$ and let $n\geq1$ be an integer. 
For $\bm h = (h_1,\ldots,h_n) \in \dH^n$, $u,v \in \dR \times \dH$ with $|u|,|v|\leq \lambda$, we define
\begin{equation}\label{def:hhat}
  \hat{h} = \hat{h}(\bm h,u,v,n) := -\left(z+\frac{1+u}{n}\sum_{i= 1}^n h_i + v \right)^{-1} \in \dH.   
\end{equation}
We define the random variable $h_r \in \dH$ with parameter $(\bm h,u,v,n)$ as follows:
\begin{equation}\label{def:hr}
h_r =
\begin{cases}
h_i &\textmd{ with } \dP(h_r = h_i) = 1/(2n), \, i = 1,\ldots,n , \\
\hat{h} &\textmd{ with } \dP(h_r = \hat{h}) = 1/2,
\end{cases}
\end{equation}
and fix a deterministic variable $h_l \in \dH$.

\begin{proposition}[Asymmetric uniform contraction]\label{prop:ACE} Let $0< E <2$ and $p \geq 2$. There exist constants $\veps,\lambda_0 $ in $(0,1)$ and a function $R\in \LIP(\lambda_0)$ depending only on $(E,p)$ such that, if $\lambda \leq \lambda_0$,
$$\dE \left[\gamma^{p} \left(\frac{h_l + h_r }{2}\right)\right] \leq (1-\veps) \left(\frac{\gamma^{p}(h_l) + \frac{1}{n}\sum_{i=1}^n\gamma^{p}(h_i)}{2}\right) + R(\lambda),$$
where the expectation is over the law of $h_r$ defined in \eqref{def:hr}. This holds uniformly in $n\geq 1$, $(\bm h,h_l) \in \dH^{n+1}$, $u,v \in \dR\times \dH$ with $|u|,|v| \leq \lambda$ and $z \in \dH_E$.
\end{proposition}

Before going into the proof, we introduce some notation. For $f \in \dC^n$, let $\dEs f$ and $\VARs f$ denote its average and variance, i.e $$\dEs f  = \frac{1}{n} \sum_{i=1}^n f_i \quad \hbox{ and } \quad \VARs f  = \frac{1}{n} \sum_{i=1}^n |f_i - \dEs f |^2.$$ 
We set $h_s = \dEs \bm h$. 
We simply write $\gamma_x$ for $\g{h_x}$ where $x\in \{1,\ldots,n\} \cup \{r,l,s\}$ and $\hat{\gamma}=\gamma(\hat{h})$.  For example, with our notation, we have 
$$
\gamma_s = \gamma(h_s) = \gamma \left(\frac 1 n \sum_i h_i \right) \quad \hbox{ and } \quad \dE_s \gamma_i = \frac 1 n \sum_{i=1}^n \gamma_i = \frac 1 n \sum_{i=1}^n \gamma(h_i).
$$
We highlight the fact that only $h_r$ and $\gamma_r$ are random. We also introduce the following quantities in $[0,1]$, with the convention $0/0 = 0$:
$$Q_s = \frac{\gamma_s^{p/2}}{\dEs \gamma_i^{p/2}}, \qquad Q_2 = \frac{\dEs \gamma_i^{p/2}}{\sqrt{\dEs \gamma_i^{p}}}, \qquad Q_l = \frac{2\sqrt{\gamma_l^{p} \dEs\gamma^{p}}}{\gamma_l^{p} + \dEs \gamma_i^{p}}.$$
We also introduce two quantities $\beta_l,\beta_2$ related to the $Q$’s that measure how $( \gamma_1,\ldots,\gamma_n)$ and $\gamma_l$ concentrate around the same value.

\begin{lemma}\label{prop:beta} Let $\beta_2 = \sqrt{1 - Q_2^2}$ and $ \beta_l = \sqrt{1 - Q_l^2}$. We have
\begin{equation}\label{eq:B2}
    \VARs \gamma_i^{p/2} \leq \beta_2^2 \dEs \gamma_i^{p}.
    \end{equation}
Furthermore, for $a  = \gamma_l^{p}$ or $a =\dEs \gamma_i^{p}$,
\begin{equation}\label{eq:Bl}
(1-  \beta_l) \frac{\gamma_l^{p} + \dEs \gamma_i^{p}}{2} \leq a \leq  (1+\beta_l) \frac{\gamma_l^{p} + \dEs \gamma_i^{p}}{2}. 
\end{equation}
\end{lemma}

\begin{proof} Equation (\ref{eq:B2}) is straightforward. To prove Equation (\ref{eq:Bl}), we note that $Q_l$ is of the form $2\sqrt{ab}/(a+b)$. Hence,
\begin{align*}
    Q_l^2 & = \frac{(a+b)^2 - (a-b)^2}{(a+b)^2} = 1-\left(\frac{a-b}{a+b}\right)^2 .
    \end{align*}
    We deduce that $   |a-b| = (a+b)\sqrt{1-Q_l^2}$ and thus
    \begin{align*}
    \max\{a,b\} &= \frac{a+b}{2} + \frac{|a-b|}{2} = \frac{a+b}{2}\left(1+\sqrt{1-Q_l^2}\right) = (1- \beta_l) \frac{a+b}{2}.\\ \intertext{Similarly,}
    \min\{a,b\} &= \frac{a+b}{2}\left(1-\sqrt{1-Q_l^2}\right) = (1- \beta_l) \frac{a+b}{2}, 
\end{align*}
as requested.
\end{proof}

Our next lemma uses the convexity to express $\dE \G{ (h_r+h_l)/2}$ in terms of the $Q$'s.
\begin{lemma}[Convexity]\label{prop:convexity} 
Let $0< E < 2$, $p \geq 2$ and $c (\cdot) \in \LIP(1)$ be as in Lemma  \ref{prop:per1}. We have
\begin{equation*}
    \dE \G{\frac{h_r+h_l}{2}} \leq (1+2 c(\lambda)) \left(\frac{3+Q_sQ_2Q_l}{4}\right) \frac{\gamma_l^{p} + \dEs \gamma_i^{p}}{2}  + c(\lambda).
\end{equation*}
\end{lemma}

\begin{proof} From Lemma \ref{prop:propgamma}$(i)$, we have $\hat{\gamma} = \g{ - 1/(z+(1+u)h_s+ v) } \leq \g{(1+u)h_s+ v}$ which combined with Lemma \ref{prop:per1}$(i)$ leads to  $\hat{\gamma}^q \leq (1+c(\lambda))\gamma_s^q + c(\lambda)$, $q \in \{ p/2,p\}$. From the definition of the $Q$'s we have 
$$\gamma_s^{p/2} = Q_sQ_2 \sqrt{\dEs \gamma_i^{p}} \quad \hbox{ and } \quad \gamma_s^{p/2} \gamma_l^{p/2} = Q_sQ_2Q_l \frac{\gamma_l^{p} + \dEs \gamma_i^{p}}{2}.$$ We get the two bounds:
\begin{align}
    \dE \gamma_r^{p} &\leq \frac{\dEs \gamma_i^{p}}{2} 
    + \frac{(1+c(\lambda))Q_s^2Q_2^2\dEs \gamma_i^{p} + c(\lambda)}{2} \leq (1+c(\lambda))\dEs \gamma_i^{p} +\frac{ c(\lambda)}{2} \label{eq:Egrr}
    \end{align}
    and 
    \begin{align}
    \gamma_l^{p/2}\dE \gamma_r^{p/2} &\leq \gamma_l^{p/2} \frac{ \dEs \gamma_i^{p/2}}{2}+\gamma_l^{p/2}\frac{(1+c(\lambda))\gamma_s^{p/2}+ c(\lambda) }{2} \nonumber \\
    & \leq (1+c(\lambda))\left(\frac{1+Q_sQ_2Q_l}{2}\right) \frac{\gamma_l^{p} + \dEs \gamma_i^{p}}{2}+ \frac{c(\lambda)(\gamma_l^{p}+1)}{4}, \label{eq:Egrr2}
\end{align}
where we used the fact that the $Q$’s are less than one, the AM-GM inequality and $2 \gamma_l^{p/2} \leq \gamma_l^{p} +1$. Now from the convexity of $x \mapsto x^{p/2}$, Lemma \ref{prop:propgamma}$(ii)$, we have 
\begin{align*}
  &\dE \gamma^{p} \left(\frac{h_r+h_l}{2}\right) \leq \dE \left(\frac{\gamma_l^{p/2} + \gamma_r^{p/2}}{2}\right)^{2} \leq  \frac{1}{2}\frac{\gamma_l^{p} +\dE \gamma_r^{p}}{2} + \frac{1}{2} \gamma_l^{p/2} \dE\gamma_r^{p/2} \\ 
  &\leq \frac{1}{2}\left((1+c(\lambda))\frac{\gamma_l^{p} + \dEs \gamma_i^{p}}{2} +\frac{c(\lambda)}{4} \right)+\frac{1}{2}\left((1+c(\lambda))\left(\frac{1+Q_sQ_2Q_l}{2}\right)\frac{\gamma_l^{p} + \dEs \gamma_i^{p}}{2} +\frac{c(\lambda)(\gamma_l^{p} +1 )}{4} \right)\\  
&\leq \left(  (1+c(\lambda)) \left( \frac{ 3  + Q_sQ_2Q_l}{4} \right) + \frac{c(\lambda)}{2} \right) \frac{\gamma_l^{p} + \dEs \gamma_i^{p}}{2} + \frac{c(\lambda)}{4}.
\end{align*}
The conclusion follows. 
\end{proof}

Lemma \ref{prop:convexity} already implies that if $Q_s Q_2 Q_l \leq 1-\veps$ then  the conclusion of Proposition \ref{prop:ACE} holds. We will now argue that if $Q_s Q_2 Q_l$ is large then a contraction occurs provided that $(\gamma_l^{p} + \dEs \gamma_i^{p})/2$ is large enough. This source of contraction will come from a default of alignment of $h_r- \Gamma$ and $h_l - \Gamma_l$ which we now explain precisely.  With the convention that $\arg 0 = \pi/2$, let $$
\alpha_{rl} = \arg{(h_r-\Gamma)\overline{(h_l-\Gamma)}}.$$
For $\alpha \in \mathbb{S}^1$, we set $\COS \hspace{-2pt} \alpha = \max\{\cos \alpha,0\}$.

\begin{lemma}[Non-alignment]\label{prop:non-alignment} 
Let $0< E < 2$, $p \geq 2$. There exist $\delta,\lambda_0>0$ and $R_0 \in \LIP(\lambda_0)$ depending only on $(E,p)$ such that if $\lambda < \lambda_0$ the following holds:
\begin{equation*}
\textmd{ if } Q_sQ_2Q_l \geq 1/2 \textmd{ and } \frac{\gamma_l^{p} + \dEs \gamma_i^{p}}{2} \geq R_0(\lambda) \; \textmd{ then } \; \dE \COS \alpha_{rl} \leq 1-\delta.
\end{equation*}
\end{lemma}
\begin{proof}
If $h_s = \Gamma$ or $h_l = \Gamma$ then $\alpha_{rl} = \pi/2$ with probability at least $1/2$ and the statement follows with $\delta \leq 1/2$. We thus assume $h_s \ne \Gamma$ and $h_l \ne \Gamma$. Let $1 \le i \le n$,  $\alpha_{i},\alpha_{s},\hat \alpha$ be defined in the same way that $\alpha_{rl}$ with $h_i,h_s,\hat h$ instead of $h_r$. Using the fact that $-\hat{h}^{-1}+\Gamma^{-1}=(1+u)h_s +  v - \Gamma$, we get, 
$$(\hat{h}-\Gamma)  = (1+u)\hat{h}\Gamma(h_s-\Gamma) \left(1+ \frac{v+u\Gamma}{(1+u)(h_s-\Gamma)}\right).$$
By setting $\theta= \arg \hat{h}\Gamma$ and $\alpha_e = \arg  \{1+(v+u\Gamma)((1+u)(h_s-\Gamma))^{-1}\}$, the above relation multiplied both sides by $\overline{(h_l-\Gamma)}$ translates into
\begin{equation}
\label{eq:alphatheta}
\hat{\alpha} = \theta + \alpha_{s} + \alpha_e,
\end{equation}
where the sum must be seen in $\mathbb{S}^1$. We take the argument in $(-\pi,\pi]$. Since $0 < E < 2$, there exists $0 < \theta_0 \le \pi/2$ such that for all $z \in \dH_E$, 
$$\arg \Gamma(z) \in (\theta_0 , \pi-\theta_0).$$ In particular, since $\hat h \in \dH$, we find $|\theta| \geq \theta_0$. We will show that for some $\delta_s>0$ to be fixed later on, under the additional hypothesis that $\dEs \COS \alpha_{i} \geq 1-\delta_s$, we have 
\begin{equation}
\label{eq:alphase}
\cos \alpha_{s} \geq \cos(\theta_0/4)\quad \hbox{ and } \quad \cos \alpha_{e} \geq \cos(\theta_0/4).
\end{equation}
This implies from \eqref{eq:alphatheta} that $|\hat{\alpha}| \geq \theta - 2 \theta_0/4 \geq \theta_0/2$. Considering the two cases $\dEs \COS \alpha_{i} \geq 1-\delta_s$ and $\dEs \COS \alpha_{i} < 1-\delta_s$, we obtain
$$\dE \COS \alpha_{rl} = \frac{1}{2} \COS \hat \alpha + \frac 1 2 \dEs \COS\alpha_i  \leq \frac{1}{2} + \frac{\max{\{1-\delta_s,\cos \theta_{0}/2\}}}{2} = 1-\delta.$$
We get the statement of Lemma \ref{prop:non-alignment} for this choice of $\delta$. It thus remain to prove that \eqref{eq:alphase} holds if $\dEs \COS \alpha_{i} \geq 1-\delta_s$.

\noindent{\em{Bound on $\alpha_{s}$. }} We have
\begin{align*}
\cos \alpha_{s} & = \frac{\Re \left[(h_s-\Gamma)\frac{\overline{(h_l-\Gamma)}}{|h_l-\Gamma|} \right] }{|h_s - \Gamma|} = \frac{1}{n} \sum_{i=1}^n \frac{ |h_i-\Gamma|\cos \alpha_{i}}{|h_s - \Gamma|} = \frac{1}{n}\sum_{i =1}^{n}\frac{  \sqrt{\gamma_i \Im h_i}\cos \alpha_{i}}{\sqrt{\gamma_s \Im h_s }}  = \sqrt{\frac{n}{\gamma_s}} \dE_s \sqrt{q_i\gamma_i} \cos\alpha_{i},
\end{align*}
where we set $q_i = \Im h_i /(\sum \Im h_j)$ already appeared in Lemma \ref{prop:propgamma}$(iii)$. From this Lemma we get 
$$\left(\sqrt{\frac{n}{\gamma_s}} \dE_s \sqrt{q_i\gamma_i}\right)^2 \geq \frac{n}{\gamma_s} \frac{1}{n^2} \sum_{i,j} \cos \alpha_{ij}\sqrt{q_iq_j \gamma_i\gamma_j} = 1.$$
Let $c = \cos^2(\theta_0/8) =\frac{1+\cos(\theta_0/4)}{2} > \cos(\theta_0/4)$ and suppose $\dEs \COS \alpha_{l}\geq 1-\delta_s$. Using that $\dEs q_i = 1/n$ and H\"older inequality with $q$ such that $\frac{1}{2}+\frac{1}{2p}+\frac{1}{q}=1$, we get 
\begin{align*}
  \sqrt{\frac{n}{\gamma_s}} \dE_s \sqrt{q_i\gamma_i\mathbf{1}_{\{\cos \alpha_{i}<c\}}}  &\leq   \sqrt{\frac{n}{\gamma_s}} \left(\dE_s q_i \right)^{1/2} \left(\dEs\gamma^{p}\right)^{1/2p}\dP_s (\cos \alpha_{i}<c)^{1/q}  
  \\
  &= \left(\frac{\sqrt{\dEs\gamma^{p}}}{\gamma_s^{p/2}}\right)^{1/p} \dP_s (\cos \alpha_{i}<c)^{1/q}\\
  &= \frac{1}{(Q_sQ_2)^{1/p}}\dP_s (1 - \COS \alpha_{i} > 1 -c )^{1/q}\\
  &\leq 2^{1/p} \left(\frac{\delta_s}{\sin^2 (\theta_0/8)}\right)^{1/q},
\end{align*}
where at the last line, we have used Markov inequality and the assumption $Q_sQ_2 \geq Q_sQ_2Q_l \geq 1/2$. We now use the bound $\cos \alpha_{i} \geq c - 2\mathbf{1}_{\{\cos \alpha_{i}<c\}}$ and obtain
$$\cos \alpha_{s} \geq \frac{1+\cos(\theta_0/4)}{2} -2^{1+1/p} \left(\frac{\delta_s}{\sin^2 (\theta_0/8)}\right)^{1/q} \geq \cos(\theta_0/4),$$
provided that $\delta_s$ is small enough. 

\noindent{\em{Bound on $\alpha_{e}$. }} Using $\Re(1 + z) / |1 + z | \geq 1 - 2 |z|$, $|\Gamma| \leq 1$ and $|u|,|v| \leq \lambda$, we have
$$\cos \alpha_e = \frac{\Re \left(1+\frac{v+u\Gamma}{(1+u)(h_s-\Gamma)}\right)}{\left |1+\frac{v+u\Gamma}{(1+u)(h_s-\Gamma)}\right|} \geq 1 - 2 \left|\frac{v+u\Gamma}{(1+u)(h_s-\Gamma)}\right| \geq 1 -   \frac{8\lambda}{|h_s-\Gamma|},$$
where we assumed without loss of generality that $\lambda \leq \lambda_0 \leq 1/2$ in the last inequality. We want to lower bound $|h_s-\Gamma|$ in terms of $\gamma_s= \gamma(h_s)$. This is easily done from the following observation. Let $r\geq 0$ and define 
$$\psi(r)=\inf \{|h-\Gamma| :   h \in \mathbb{H},\, \g{h}\geq r \}.$$
A drawing shows that this infimum is attained at the unique $h^*$ such that $\Re h^* = \Re \Gamma$, $0<\Im h^* < \Im \Gamma$ and $\g{h^*}=r$. For this configuration, we have $\Im \Gamma = \Im h^* + |h^*-\Gamma|$. Since $\psi(r)=|h^*-\Gamma|$, we get 
$$r = \g{h^*} = \frac{|h^*-\Gamma|}{\Im\Gamma \Im h^*} = \frac{\psi^2(r)}{\Im \Gamma (\Im \Gamma-\psi(r))}.$$
Thus $\psi^{-1}$ is defined on $[0,\Im \Gamma)$ and $\psi^{-1}(t) = t^2 / (\Im \Gamma (\Im \Gamma-t))$. 
As $Q_sQ_2 \geq 1/2$ and $Q_l\geq 1/2$, from \eqref{eq:Bl} in Lemma \ref{prop:beta}, we have $1-\beta_l \geq 1/8$ and
$$\gamma_s^{p} = (Q_s Q_2)^2 \dEs \gamma_i^{p} \geq \frac{1}{32}\frac{\gamma_l^{p} + \dEs \gamma_i^{p}}{2}.$$
Let $R_0(\lambda)= 32 \left(\psi^{-1} \left( 8 \lambda / (1-\cos(\theta_0/4)) \right)\right)^{p}$ which is defined for any $z\in \dH_E$ as soon as
$$\lambda < \lambda_0 := \frac{1-\cos(\theta_0/4)}{8} \min_{z\in \dH_E} \Im \Gamma(z).$$
We have $\cos \alpha_e \geq \cos(\theta_0/4)$ as soon as $(\gamma_l^{p} + \dEs \gamma_i^{p})/2 \geq R_0(\lambda)$. Since $\psi^{-1}(t)\sim t^2$ as $t$ goes to $0$, we have $R_0 \in  \LIP(\lambda_0)$. This concludes the proof.
\end{proof}

The next lemma expresses the contraction in terms of $\COS\alpha_{rl}$. Recall the definition $\beta_2,\beta_l$ from Lemma \ref{prop:beta}.
\begin{lemma}[Contraction from non-alignment]\label{prop:contraction}
Let $0< E < 2$, $p \geq 2$ and $c (\cdot) \in \LIP(1)$ be as in Lemma  \ref{prop:per1}. We have
\begin{equation*}\label{eq:nonalignimplies}
\dE \G{\frac{h_r + h_l}{2}} \leq (1+2c(\lambda))(1+\beta_l)\left(\frac{3+\dE\COS \alpha_{rl}}{4} + \beta_2  \right)\frac{\gamma_l^{p} + \dEs \gamma_i^{p}}{2}  + c(\lambda).
\end{equation*}
\end{lemma}

\begin{proof} Let $p_r =  \Im h_r / (\Im  h_r + \Im  h_l)$ from Lemma \ref{prop:propgamma}$(iii)$ and the AM-GM inequality, we have
\begin{align*}
  \g{\frac{h_r + h_l}{2}} &
  = \frac{1}{2}p_r \gamma_r + \frac{1}{2}(1-p_r) \gamma_l + \cos \alpha_{rl} \sqrt{p_r (1-p_r)}\sqrt{\gamma_r\gamma_l}\\
  &\leq \frac{1}{2}p_r \gamma_r + \frac{1}{2}(1 -p_r) \gamma_l + \frac{1}{4}\COS (\alpha_{rl} )(\gamma_r+\gamma_l).
\end{align*}
Using the convexity of $x^{p/2}$ and that $\COS^{p/2} \leq \COS$ we get 
$$\gamma^{p/2} \left(\frac{h_r + h_l}{2}\right) \leq \frac{1}{2}p_r \gamma_r^{p/2} + \frac{1}{2}(1 -p_r) \gamma_l^{p/2} + \frac{1}{4}\COS (\alpha_{rl} )(\gamma_r^{p/2}+\gamma_l^{p/2}).$$
We treat the two events $\{h_r = h_i \hbox{ for some } i\in \{1,\ldots,n\}\}$ and $\{h_r = \hat{h}\}$ separately. For the first one, we write, with $p_i = \Im h_i / (\Im h_r + \Im h_l)$ and $\alpha_i$ as in Lemma \ref{prop:non-alignment}, 
\begin{align*}
  \dEs \gamma^{p/2}\left(\frac{h_i+ h_l}{2}\right)&\leq \frac{1}{2}\dEs[p_i\gamma_i^{p/2}] + \frac{1}{2}(1-\dEs[p_i]) \gamma_l^{p/2} + \frac{1}{4}\dEs[\COS( \alpha_{i}) (\gamma_i^{p/2}+\gamma_l^{p/2})]\\
  &\leq \frac{1}{2}\dEs[p_i] \dEs[
  \gamma_i^{p/2}] + \frac{1}{2}(1-\dEs[p_i]) \gamma_l^{p/2} + \frac{1}{2}\dEs[\COS \alpha_{i}]\left(\frac{\dEs[\gamma_i^{p/2}]+\gamma_l^{p/2}}{2}\right) \\
  & \quad \quad \quad + (\frac{1}{2}+\frac{1}{4}) \sqrt{\VARs \gamma_i^{p/2}}\\
  &\leq \left(\frac{1 + \dEs \COS \alpha_{i}}{2} + \beta_2 \right) \sqrt{1+ \beta_l} \sqrt{ \frac{\gamma_l^{p} + \dEs \gamma_i^{p}}{2}},
\end{align*}
where we used Jensen inequality, \eqref{eq:B2}-\eqref{eq:Bl} from Lemma \ref{prop:beta} together with $\dEs [f_i\gamma_i^{p/2}] \leq \dEs[f_i]\dEs[\gamma_i^{p/2}] + \sqrt{\VARs \gamma_i^{p/2}}$ if $0 \leq f_i \leq 1$ (for $f_i = \COS \alpha_{i}$ and $f_i=p_i$).

Now, we deal with event $\{h_r = \hat{h}\}$. From Lemma \ref{prop:convexity}$(i)$ and Lemma \ref{prop:per1}$(i)$, we have $\hat{\gamma}^{p/2} \leq (1+c(\lambda)) \dE_s \gamma_i^{p/2} + c(\lambda)$. We deduce, with $\hat p = \Im \hat h / (\Im \hat h + \Im h_l)$,
\begin{align*}
  \gamma^{p/2} \left(\frac{h_r + h_l}{2}\right) &\leq \frac{1}{2}\hat{p} \hat{\gamma}^{p/2} + \frac{1}{2}(1-\hat{p}) \gamma_l^{p/2} + \frac{1}{4}\COS \hat \alpha(\hat{\gamma}^{p/2}+\gamma_l^{p/2})\\
  & \leq (1+c(\lambda))\left(\hat{p} \dE_s \gamma_i^{p/2} + (1-\hat{p})\gamma_l^{p/2} + \frac{\COS \hat \alpha}{2}\frac{ \dE_s \gamma_i^{p/2} + \gamma_l^{p/2}}{2}\right) + (\frac{1}{2}+\frac{1}{4})c(\lambda)\\
  & \leq (1+c(\lambda))\left(\frac{1+\COS \hat \alpha}{2}\right)\sqrt{1+\beta_l}\sqrt{\frac{\gamma_l^{p} + \dEs \gamma_i^{p}}{2}} +  c(\lambda),
\end{align*}
where have used Jensen inequality and Lemma \ref{prop:beta}. Combining the two events and multiplying by $\gamma_l^{p/2}$ we get from a new application of Lemma \ref{prop:beta} and Jensen inequality,
\begin{align*}
\gamma_l^{p/2} \dE \gamma^{p/2}\left(\frac{h_r + h_l}{2}\right) &\leq ( 1+ c(\lambda)) \left(\frac{1 + \dE \COS \alpha_{rl}}{2} + \beta_2 \right)(1+\beta_l) \frac{\gamma_l^{p} + \dEs \gamma_i^{p}}{2} + \frac{c(\lambda)}{2}\gamma_l^{p/2} \\
& \leq ( 1+ 2 c(\lambda)) \left(\frac{1 + \dE \COS \alpha_{rl}}{2} + \beta_2 \right)(1+\beta_l) \frac{\gamma_l^{p} + \dEs \gamma_i^{p}}{2} + \frac{c(\lambda)}{4}.
\end{align*}

Moreover, from \eqref{eq:Egrr}-\eqref{eq:Egrr2} we obtain
\begin{align*}
    \dE \gamma_r^{p} &\leq (1+c(\lambda))\dEs \gamma_i^{p} +\frac{ c(\lambda)}{2} \leq  (1+c(\lambda))(1 + \beta_l) \left(\frac{\gamma_l^{p} + \dEs \gamma_i^{p}}{2}\right) +\frac{ c(\lambda)}{2}  \\
    \gamma_l^{p/2}\dE \gamma_r^{p/2} & \leq (1+c(\lambda)) \frac{\gamma_l^{p} + \dEs \gamma_i^{p}}{2}+ \frac{c(\lambda)(\gamma_l^{p}+1)}{4} \leq (1+2 c(\lambda)) (1 + \beta_l) \frac{\gamma_l^{p} + \dEs \gamma_i^{p}}{2} + c(\lambda)/4.  
\end{align*}

Finally, using Lemma \ref{prop:propgamma}$(ii)$ at the first line and then applying the above bounds, we find
\begin{align*}
    \mathbb{E}\G{\frac{h_r + h_l}{2}} \leq & \frac{1}{4}\mathbb{E}\gamma_r^{p} + \frac{1}{4}\mathbb{E}[\gamma_r^{p/2}]\gamma_l^{p/2}  + \frac{\gamma_l^{p/2}}{2}\mathbb{E} \gamma^{p/2}\left(\frac{h_r + h_l}{2}\right)\\
    \leq & (1+2c(\lambda))(1+\beta_l)\left[ \frac{1}{4} + \frac{1}{4} + \frac{1}{2}\left(\frac{1+\mathbb{E}\COS \alpha_{rl}}{2} + \beta_2 \right) \right] \left(\frac{\gamma_l^{p} + \dEs \gamma_i^{p}}{2}\right) + c(\lambda).
\end{align*}
This concludes the proof.
\end{proof}

All ingredients are now gathered to establish the asymmetric uniform contraction. 
\begin{proof}[Proof of Proposition \ref{prop:ACE}] Let $\delta, \lambda_0, R_0(.)$ from Lemma \ref{prop:non-alignment}. There exists $\veps_0 = \veps_0(E,p) >0$ such that 
$$Q_sQ_2Q_l \geq 1-\veps_0 \implies (1+\beta_l)\left(1 - \frac \delta 4+ \beta_2  \right) \leq 1-\frac \delta 5.$$
Without loss of generality, we suppose that $\veps_0 \leq 1/2$. We may also assume that $\lambda_0$ is small enough so that
$$(1+2 c(\lambda_0) )\max\left\{1-\frac{\veps_0}{4},1-\frac{\delta}{ 5}\right\} = 1-\veps,$$
for some $\veps=\veps(E,p)>0$. Now if $(\gamma_l^{p} + \dEs \gamma_i^{p})/2 \leq R_0(\lambda)$, then by Lemma \ref{prop:convexity} we have 
$$\dE \G{\frac{h_r+h_l}{2}} \leq (1+2 c(\lambda)) R_0(\lambda)  + c(\lambda) := R(\lambda).$$
Note that by definition $R \in \LIP(\lambda_0)$. If $Q_sQ_2Q_l \leq 1- \veps_0$ then by Lemma \ref{prop:convexity} again we have 
$$\dE \G{\frac{h_r+h_l}{2}} \leq (1+2 c(\lambda)) \left(1-\frac{\veps_0}{4}\right) \frac{\gamma_l^{p} + \dEs \gamma_i^{p}}{2}  + c(\lambda).$$
Otherwise $(\gamma_l^{p} + \dEs \gamma_i^{p})/2 \geq R_0( \lambda)$ and  $Q_sQ_2Q_l \geq 1-\veps_0 \geq 1/2$. We may use Lemma \ref{prop:non-alignment} together with Lemma \ref{prop:contraction} and get
$$\dE \G{\frac{h_r + h_l}{2}} \leq (1+2 c(\lambda))\left(1-\frac \delta 5\right)\frac{\gamma_l^{p} + \dEs \gamma_i^{p}}{2}  + c(\lambda).$$
From our choice of $\veps>0$ and $R(\lambda)$, we find in that all cases
$$\dE \G{\frac{h_r + h_l}{2}} \leq (1-\veps)\frac{\gamma_l^{p} + \dEs \gamma_i^{p}}{2}  + R(\lambda),$$
which concludes the proof.
\end{proof}

\subsection{Proof of Theorem \ref{thm:main}}

\noindent{\em Step 1: rough bound on $\g{g}$. } We assume that $N$ has finite $2p$-moments. We first check that $\dE \G{g}$ is finite for any $z  \in \dH_E$. From \eqref{eq:RDEz},
$$\Im g \stackrel{d}{=}  \frac{\Im ( z + \frac{1}{d}\sum_{i=1}^{N} g_i + v)  }{|z +\frac{1}{d}\sum_{i=1}^{N} g_i + v|^2} \geq \frac{\Im z }{|z+\frac{1}{d}\sum_{i=1}^{N} g_i +v|^2} \geq \frac{\Im z }{\left(|z|+\frac{N}{d \Im z} +|v| \right)^2} .$$
Moreover $|g|,|g_i|,|\Gamma| \leq 1/\Im z$. We get  
$$\dE \G{g} = \dE \frac{|g-\Gamma|^{2p}}{(\Im \Gamma \Im g)^{p}} \leq \frac{1}{(\Im \Gamma)^p}\left(\frac{2}{ \Im z}\right)^{2p} \frac{1}{(\Im z)^{p}}\dE \left( |z|+\frac{N}{d\Im z} +|v| \right)^{2p}< \infty,$$
since by assumption both $N$ and $v$ admit $2p$-moments. Now as the expectation is finite, we will show that, uniformly in $z\in \dH_E$, $\dE \G{g} \leq (1-\veps) \dE \G{g} +R$ to conclude that $\dE \G{g} \leq R/\veps$. 

\noindent{\em Step 2: averaging with two copies. } We now upper bound $\dE\G{g}$ in terms of $\dE\G{(g_r +g_l)/2}$ where $g_r,g_l$ are two independents copies of $g$.  On the event $\cE_\lambda$ defined in \eqref{eq:defEl} we have $N \geq 2$. We use Lemma \ref{prop:propgamma} and Lemma \ref{prop:per1}$(i)$ and obtain 
\begin{align*}
  \mathbf{1}_{\cE_\lambda}\G{g} &\leq \mathbf{1}_{\cE_\lambda}\G{\frac{N}{d}\frac{1}{N}\sum_{i=1}^{N} g_i + v } \leq \mathbf{1}_{\cE_\lambda}(1+c(\lambda))\G{\frac{1}{N}\sum_{i=1}^{N} g_i}+\mathbf{1}_{\cE_\lambda}c(\lambda)\\ 
  &\leq  (1+c(\lambda))\frac{1}{N(N -1)}\sum_{1 \leq i \ne j  \le N} \G{\frac{g_i + g_j}{2} }+ c(\lambda).
\end{align*}
By taking expectation over the $g_i$’s and use that given $(N,v)$, the $g_i$'s are iid copies of $g$, we get
$$  \dE \mathbf{1}_{\mathcal{E}_\lambda}\G{g} \leq  (1+c(\lambda)) \dE \G{ \frac{g_r + g_l}{2}}  +  c(\lambda).$$
On the complementary event $\cE_\lambda^c$, we use Lemma \ref{prop:per1}$(iii)$:   
\begin{align*}
  \mathbf{1}_{\cE^c_\lambda}\G{g_o} &\leq \mathbf{1}_{\cE^c_\lambda}\G{ \frac{N}{d}\frac{1}{N}\sum_{i=1}^{N} g_i + v } \leq  \mathbf{1}_{\cE^c_\lambda} C(1+|v|)^{2p} \left(\frac{N}{d}+\frac{d}{N}\right)^{p}\left(\G{\frac{1}{N}\sum_{i=1}^{N} g_i } +1\right)\\
  & \leq  \mathbf{1}_{\cE^c_\lambda} C(1+|v|)^{2p} \left(\frac{N}{d}+\frac{d}{N}\right)^{p}\frac{1}{N}\sum_{i=1}^{N} \left(\G{ g_i } +1\right).
\end{align*}
We average first over the $g_i$’s and obtain  
$$  \dE \mathbf{1}_{\cE^c_\lambda}\G{g} \leq C \dE \left[\mathbf{1}_{\cE^c_\lambda}(1+|v|)^{2p} \left(\frac{N}{d}+\frac{d}{N}\right)^{p} \right] \left(\dE\G{ g} +1\right) = C \alpha_p(\lambda) \left(\dE\G{ g } +1\right).$$

From the identity $1 = \mathbf{1}_{\mathcal{E}_\lambda} + \mathbf{1}_{\cE^c_\lambda}$, we obtain:
\begin{equation} \label{eq:step2}
      \dE \G{g} \leq (1+c(\lambda)) \dE  \G{\frac{g_r + g_l}{2}}  +  c(\lambda) + C \alpha_p(\lambda)\left(\dE\G{ g} +1\right).
\end{equation}

\noindent{\em Step 3: asymmetric uniform contraction. } We now bound $\dE   \G{(g_r +g_l)/2}$ in terms of $\dE\G{g}$. There exists an enlarged probability space for the variables $(g_r,g_l)$ such that 
$$
g_r = - \left( z + \frac 1 d \sum_{i=1}^{N} h_i + v \right)^{-1},
$$
where $(h_i)_{i \geq 1}$ are iid with law $g_r \stackrel{d}{=} g$ and independent of $(N,v)$.  
For integer $n \geq 1$ and $\lambda \in [0,1)$, let $\mathcal{E}_{\lambda,n} = \{|v| \leq \lambda ,  N = n\}$. We have $\mathcal{E}_\lambda = \bigcup_{n} \mathcal{E}_{\lambda,n}$ where the union is over all $n \geq 2$ such that $n/d \in (1-\lambda,1+\lambda)$. On the event $\cE_{\lambda,n}$, we set $u_n = n/d -1 \in (-\lambda,\lambda)$ and $\bm h = (h_1,\ldots,h_n)$. Hence, on the event $\cE_{\lambda,n}$, we have
$$
g_r = - \left( z + \frac {1 + u_n} d \sum_{i=1}^{n} h_i + v \right)^{-1} = \hat h (\bm h,u_n,v , n),
$$
where $\hat h$ was defined in \eqref{def:hhat}. Let $v_n$ be a random variable whose distribution is $v$ given $\cE_{\lambda,n}$ and let $h_r = h_r(\bm h,u,v,n)$ be the random variable defined in \eqref{def:hr}. Using Proposition \ref{prop:ACE}, we deduce that if $\lambda \leq \lambda_0$,
\begin{align*}
    &\frac{1}{2} \dE \left[\G{\frac{g_r +g_l}{2}}\right]+\frac{1}{2} \dE \left[ \G{\frac{g_r +g_l}{2}}\Bigg{|} \mathcal{E}_{\lambda,n} \right] \\
    & = \dE \left[ \frac{1}{2n} \sum_{i=1}^{n} \G{\frac{h_i +g_l}{2}} + \frac{1}{2}\G{  \frac{\hat h (\bm h,u_n,v_n,n) + g_l}{2}}\right]\\
    &= \dE \left[\G{  \frac{h_r(\bm h ,u_n,v_n,n) + g_l}{2}}\right]\\
    & \leq \dE\left[(1-\veps) \left(\frac 1 2 \G{g_l} + \frac{1}{2n}\sum_{i=1}^n\G{h_i}\right) + R(\lambda)\right]\\
    &= (1-\veps)\dE [\G{g}]+R(\lambda).
\end{align*}

We now use the identity $$1 = \frac{\dP(\mathcal{E}^c_{\lambda})+\mathbf{1}_{\mathcal{E}^c_\lambda}}{2}+ \sum_n \frac{\dP(\mathcal{E}_{\lambda,n})+\mathbf{1}_{\mathcal{E}_{\lambda,n}}}{2},$$
where the sum is over all $n \geq 2$ such that $n/d \in (1-\lambda,1+\lambda)$.
The first term is bounded using the convexity of $\gamma^p$, Lemma \ref{prop:propgamma}$(ii)$, and Lemma \ref{prop:per1}$(iii)$,
\begin{align*}
  &\dE \left[\G{\frac{g_r +g_l}{2}} \frac{\dP(\mathcal{E}^c_\lambda)+\mathbf{1}_{\mathcal{E}^c_\lambda}}{2}\right] \leq \dE \left[\frac{\G{g_r} +\G{g_l}}{2} \frac{\dP(\cE_\lambda^c)+\mathbf{1}_{\cE_\lambda^c}}{2}\right] \\
  &\leq\frac{3}{4} \dP(\cE_\lambda^c) \dE \G{g}+ \frac{C}{4} \dE \left[\mathbf{1}_{\cE^c_\lambda}(1+|v|)^{2p} \left(\frac{N}{d}+\frac{d}{N}\right)^{p} \right] \left(\dE \G{g} +1\right) \\
  &\leq \frac{C+3}{4} \alpha_p(\lambda) \left(\dE \G{g} +1\right) ,
\end{align*}
where we used the fact that $\dP(\cE^c_\lambda)\leq  \alpha_p(\lambda)$.
Finally we obtain
\begin{equation}\label{eq:step3}
    \dE \G{\frac{g_r +g_l}{2}}  
    \leq (1-\veps)\dE  \G{g} +R(\lambda) + \frac{C+3}{4}\alpha_p(\lambda) \left(\dE \G{g} +1\right).
    \end{equation}

\noindent{\em Step 4: end of proof. } 
If $\lambda$ is small enough so that $1+c(\lambda) \leq 2$, combining \eqref{eq:step2}-\eqref{eq:step3}, we get
$$\dE \G{g} \leq (1-\veps)(1+c(\lambda))\dE \G{g} + 2R(\lambda) + c(\lambda) + (2C +1 ) \alpha_p(\lambda) \left(\dE\G{ g } +1\right).$$
We now assume that $\lambda+\alpha_p(\lambda)$ is small enough, say $\lambda_1$, so that $(1-\veps)(1+c(\lambda)) + (2 C +1) \alpha_p(\lambda) \leq 1-\veps/2$. Then
$$\dE \G{g}  \leq \left( 1-\frac \veps 2 \right)\dE \G{g} + 2R(\lambda)+c(\lambda)+(2C+1)\alpha_p(\lambda). $$
The above equation is of the form
$$\dE \G{g} \leq (1-\veps_1)\dE \G{g}+ R_1(\lambda  + \alpha_p(\lambda)),$$
where $\veps_1 = \veps /2 > 0$ and $R_1 \in \LIP(\lambda_1)$. We deduce that for all $\lambda \leq \lambda_1$,
\begin{equation} \label{eq:thmmain}
  \dE \G{g} \leq \frac{R_1 (\lambda + \alpha_p(\lambda))}{\veps_1}.  
\end{equation}
Finally, we use Lemma \ref{lem:boundgamma}  and Cauchy-Schwarz inequality,
$$
  \dE |g-\Gamma|^{p} \leq |\Gamma|^{p} 2^{p/2-1} \left( 4^{p/2} \dE \gamma^p(g) + 2^{p/2} \dE \gamma(g)^{p/2}\right) \leq |\Gamma|^{p}2^{2p} \left(\dE \G{g} +\sqrt{ \dE \G{g}}\right), 
$$
$$  \dE (\Im g)^{-p} \leq (\Im \Gamma)^{-p} \dE (4\gamma(g,h)+2)^{p} < \infty.$$
As $|\Gamma|, (\Im \Gamma)^{-1}$ are uniformly bounded in $\dH_E$, this concludes the proof of Theorem \ref{thm:main} for suitable constants $\delta = \lambda_1/2$ and $C >0$ when $N$ has finite $2p$-moments.

\noindent{\em Step 5: truncation. } It remains to check that we can remove the assumption that $\dE N^{2p}$ is finite if $v$ is real-valued. This can be done easily by truncation. From \eqref{recursion}, $g(z)$ has law $g_1(z) = \langle \delta_1 , ( H -  z)^{-1} \delta_1 \rangle$. For integer $n \geq 1$ and $v \in \mathcal V$, let $N^{(n)}_v = \max (N_v,n)$ and let $ \cT_n$ be the corresponding random rooted tree of descendants of $1$. The tree  $ \cT_n$ has law $\GW(P_n)$ where $P_n$ is the law of $N^{(n)}= \max(N,n)$ with $N \stackrel{d}{=} P$. By monotone convergence, $d_n = \dE N^{(n)}$ converges to $d = \dE N$ as $n$ goes to infinity. Let $H_n$ be the corresponding operator and $g^{(n)}_1(z) = \langle \delta_1 , (H_n -  z)^{-1} \delta_1 \rangle$. Then by \cite[Theorem VIII.25(a)]{MR0493421}, for any fixed $z \in \dH$, with probability one, 
$$
\lim_{n \to \infty} g^{(n)}_1(z) = g_1(z),
$$
(note however that the convergence is a priori not uniform in $z \in \dH$). Since $|g^{(n)}_1(z) | \leq 1 / \Im(z)$, the convergence also holds in $L^q$ for any $q\geq 1$. Note also that by monotone convergence, for any $\lambda >0$, $\alpha_p^{(n)}( \lambda)$ converges to $\alpha_p(\lambda)$ where $\alpha_p^{(n)}( \lambda)$ is defined as  in \eqref{eq:defalphal} for $P_n$ and $d_n$. It thus remains to apply Theorem \ref{thm:main} to $P_n$ and let $n$ go to infinity. 
\qed

\subsection{Proof of Corollary \ref{cor:main}}
Set $N_o = N_\star$ and $V_o = v_\star$. From \eqref{recursion}, we have 
$$
g_o = - \left( z + \frac 1 d \sum_{i=1}^{N_\star} g_i + v_* \right)^{-1},
$$
where $(g_i)_{i \geq 1}$ are independent iid copies of $g$ which satisfies the equation in distribution \eqref{eq:RDEz}. From \eqref{screcursiono}, we get
$$
\gamma( g_o,\Gamma_\star) = \gamma \left( - \frac 1  {g_o}, - \frac{1}{\Gamma_\star} \right)  \leq \gamma \left(  \frac 1 d \sum_{i=1}^{N_\star} g_i + v_* , \rho \Gamma\right) = \gamma \left(  \frac {N_\star} {d_\star} \frac{1}{N_\star} \sum_{i=1}^{N_\star} g_i + \frac{v_*}{\rho} , \Gamma\right),
$$
where we have used $\gamma(ta,tb) = \gamma(a,b)$ for $t >0$ and $\rho =  d_\star/d$. From Lemma \ref{prop:propgamma}$(ii)$ and Lemma \ref{prop:per1}$(ii)$, we deduce that 
$$
\IND_{\cF_\lambda^{ c}}\gamma^{p}( g_o,\Gamma_\star) \leq C \IND_{\cF_\lambda^{c}} \left(1+\left|\frac{v_\star}{\rho}\right|\right)^{2p}\left( \frac{N_\star}{d_\star} + \frac{d_\star}{N_\star} \right)^{p}\left(\frac 1 {N_\star} \sum_{i=1}^{N_\star} \gamma^{p}(g_i) +1 \right),
$$
$$\dE \IND_{\cF_\lambda^{c}} \gamma^{p}( g_o,\Gamma_\star) \leq C \beta_p(\lambda) \left(\dE \gamma^p(g)+1\right).$$
Similarly, from Lemma \ref{prop:propgamma}$(ii)$ and Lemma \ref{prop:per1}$(i)$,
$$
\IND_{\cF_\lambda} \gamma^{p}( g_o,\Gamma_\star) \leq (1+c(\lambda)) \frac{1}{N_\star} \sum_{i=1}^{N_\star}\gamma^{p}(g_i) +c(\lambda),
$$
$$\dE \IND_{\cF_\lambda} \gamma^{p}( g_o,\Gamma_\star) \leq (1+c(\lambda)) \dE \gamma^p(g) + c(\lambda)$$
Taking the expectation, we deduce that if $\lambda + \alpha_p(\lambda) \leq \delta$,
\begin{align*}
\dE \gamma^{p}( g_o,\Gamma_\star) &\leq (1+c(\lambda)+ C\beta_p(\lambda))\dE \gamma^{p}(g) +C\beta_p(\lambda) + c(\lambda)\\
&\leq C_1(1+\lambda+\beta_p(\lambda))(\lambda+\alpha_p(\lambda))+C_1(\beta_p(\lambda) + \lambda)\\
&\leq C_1\left((1+\lambda+\beta_p(\lambda))(1+\lambda+\alpha_p(\lambda))-1\right)
\end{align*}
for some $C_1 >0$, where we have used \eqref{eq:thmmain} and the fact that $c(\lambda) \in \LIP(\delta)$. We have 
$$
|\Gamma_{\star}| \leq \frac{1}{\Im (z  + \rho \Gamma) } \leq  \frac{1}{ \rho \Im (\Gamma) } \leq \frac{C_2}{\rho}, 
$$
$$\Im \Gamma_\star = \frac{\Im z + \rho \Im \Gamma}{|z+\rho\Gamma|^2} \geq \frac{1}{2}\frac{\rho \Im \Gamma}{|z|^2+\rho^2 |\Gamma|} \geq C_2 \left(\rho+\frac{1}{\rho}\right)^{-1},$$
uniformly in $z\in \dH_E$, for some constant $C_2$. It remains to use Lemma \ref{lem:boundgamma} and adjust the constants.
\qed
\section{Random trees with leaves}
\label{sec:leaves}

\subsection{Proof of Lemma \ref{lem:RDEleaves}}

\label{subsec:skeleton}

Let $v \ne o \in \mathcal V$. Let $\bar N^s$ and $\bar N^e$ be the number of offspring of $v $ in $\mathcal S$ and not in $\mathcal S$ (that is such that their subtree is infinite or finite).  The pair $(\bar N^s,\bar N^e)$ has the same distribution than $( \sum_{i=1} ^N ( 1 -\veps_i) , \sum_{i=1} ^N \veps_i)$ where $N$ has distribution $P$ and is independent of the $(\veps_i)_{i \geq 1}$ an iid sequence of Bernoulli variables with $\dP ( \veps_i = 1) = \pi_e = 1 - \dP ( \veps_i = 0)$. Moreover, conditioned on the root is in $\mathcal S$, $(\bar N^s,\bar N^e)$ is conditioned on $\bar N^s \geq 1$. As in Lemma \ref{lem:RDEleaves}, let $(N^s, N^e)$ denote a pair of random variables with  distribution $(\bar N^s,\bar N^e)$ conditioned on $\bar N^s \geq 1$. In particular, the moment generating function of $(N^s,N^e)$ is given by 
\begin{equation}\label{eq:varphis}
\varphi^{s,e} (x,y) = \dE\left[ x^{N^s} y ^{N^e}  \right]  =   \dE\left[ x^{\bar N^s} y ^{\bar N^e}  | \bar N^s \geq 1\right] = \frac{ \varphi ( (1 - \pi_e ) x + \pi_e y ) - \varphi( \pi_e y)  } { 1 - \varphi( \pi_e y) }. 
\end{equation}
Similarly, given $v \notin \mathcal S$, $(N^s,N^e)$ is conditioned on $N^s = 0$.  Then, we find  that the moment generating function of $N^e$ given $o \notin \mathcal S$ is 
\begin{equation}\label{eq:varphie}
\varphi^e (x) = \frac{ \varphi ( \pi_e x ) } {   \pi_e}. 
\end{equation}
For more details see Athreya and Ney \cite[Section I.12]{MR0373040} or Durrett \cite[Section 2.1]{MR2656427}.

Similarly, let $\bar N_\star^s$ and $\bar N_\star^e$ be the number of offspring of $o $ in $\mathcal S$ and not in $\mathcal S$. Let $(N_\star^s,N_\star^e)$ be the law of $(\bar N_\star^s,\bar N_\star^e)$ given $\bar N_\star^s \geq 1$. The moment generating function of $(N_\star^s,N_\star^e)$  is 
\begin{equation*}\label{eq:varphisstar}
\varphi^{s,e}_\star (x,y) = \dE\left[ x^{N_\star^s} y ^{N_\star^e}  \right] = \frac{ \varphi_\star( (1 - \pi_e ) x + \pi_e y ) - \varphi_\star( \pi_e y)  } { 1 - \varphi_\star( \pi_e y) }. 
\end{equation*}

We are ready to prove Lemma \ref{lem:RDEleaves}. 

\begin{proof}[Proof of Lemma \ref{lem:RDEleaves}]
Recall that $H = A / \sqrt d_s$ where $d_s = \dE N^s$ and $g_v(z) = \langle \delta_v , (H_v -  z)^{-1} \delta_v \rangle$. For $v \ne o$, let $g^s(z)$ and $g^e(z)$ be the law of $g_v(z)$ conditioned on $v \in \mathcal S$ and $v \notin \mathcal S$. We deduce from \eqref{recursion} that the variable $g^s(z)$ satisfies the recursive distribution equation, 
\begin{equation*}
g^s(z) \stackrel{d}{=} - \left( z + \frac 1 {d_s} \sum_{i=1}^{N^s}  g^s_i (z) +  v(z) \right)^{-1},  
\end{equation*}
where $g^s_i$ are iid copies of $g^s$, independent of $(N^s,v(z)) $  defined by 
\begin{equation}\label{eq:defV}
v(z) =  \frac 1 {d_s} \sum_{i = 1} ^{N^e} g^e_{i} (z), 
\end{equation}
and $g^e_i$ are iid copies of $g^e$ and are independent of $(N^s,N^e)$ with moment generating function given by \eqref{eq:varphis}. Similarly, conditioned on $o \in \mathcal S$, we have 
\begin{equation*}
g_o(z) \stackrel{d}{=} - \left( z + \frac 1 {d_s} \sum_{i=1}^{N_\star^s}  g^s_i (z) +   \frac 1 {d_s} \sum_{i = 1} ^{N_\star^e} g^e_{i} (z)  \right)^{-1},  
\end{equation*}
where $g^s_i$ are iid copies of $g^s$, independent of $(N_\star^s,N_\star^e) $.
It concludes the proof of Lemma \ref{lem:RDEleaves}.
 \end{proof}

\begin{remark}
As we defined it, the skeleton tree $\mathcal S$ is not a unimodular tree. To obtain a unimodular Galton-Watson tree, the solution is to modify the rule for the root to be in $\mathcal S$: the root is in $\mathcal S$ if it has at least two offsprings with infinite subtrees. 
\end{remark}

\subsection{Proof of Lemma \ref{lem:leaves}}

We start with a rough upper bound on the extinction probability $\pi_e$. 

\begin{lemma}[Extinction probability]\label{le:pie}
We have
$$
\pi_e \leq 2 \dP ( N \leq 1).
$$
\end{lemma} 
\begin{proof}
We note that if $x \in [0,1]$ satisfies $\varphi(x) \leq x$ then $\pi_e \leq x$. Set $\pi_1 = \dP ( N \leq 1)$. For $x \in [0,1]$, we have 
$\varphi(x) \leq \pi_1 + (1-\pi_1) x^2$. In particular, since $\pi_1(1-\pi_1) \leq 1/4$,
$\varphi(2\pi_1) \leq \pi_1 ( 1 + 4 (1-\pi_1) \pi_1) \leq 2\pi_1. 
$ We conclude that $\pi_e \leq 2\pi_1$. 
\end{proof}

We will also use \cite[Lemma 23]{MR3411739} on the size of the sub-critical Galton-Watson trees. 

\begin{lemma}[Total progeny of subcritical Galton-Watson tree]\label{le:totalprogeny}
Let $Q$ be a probability measure on non-negative integers whose moment generating function $\psi$ satisfies $\psi(\rho) \leq \rho$ for some $\rho \geq 1$. Let $Z$ be the total number of vertices in a $\GW(Q)$ tree. We have for any integer $k \geq 1$,
$$
\dP ( Z \geq k) \leq \rho \left( \frac{ \psi( \rho)}{\rho}\right)^k.
$$
\end{lemma}

As a corollary, we deduce that the random tree $\cT$ conditioned on extinction is likely to have very few vertices.
\begin{corollary}\label{cor:totalprogeny}
Let $n \geq 2$, $ q = \dP ( N < n)$ and let $Z$ be the total number of vertices in $\mathcal T$ conditioned on extinction. If $q \leq  2^{-n/(n-1)}$, we have for any integer $k \geq 1$,
$$
\dP ( Z \geq k) \leq \frac{ q^{1/n}}{\pi_e}  \left(2 q^{ 1 - 1/n}\right)^k.
$$
\end{corollary}
\begin{proof}
We have $\varphi( q^{1/n})  \leq q + ( q ^{1/n})^n (1 - q) \leq 2 q \leq q^{1/n} \delta$ with $\delta = 2 q ^{1-1/n}$. By assumption $\delta \leq 1$. Let $\rho = q^{1/n} / \pi_e$. From \eqref{eq:varphie}, we deduce that 
$$
\varphi_e( \rho) = \frac{\varphi( q^{1/n} )  }{\pi_e} \leq \delta \rho.
$$
Note that since $\varphi'_e(1) = \varphi'(\pi_e) < 1$ and $\varphi_e$ convex, we have $\varphi_e(x) > x$ on $[0,1)$. It follows that $\rho \geq 1$. We then conclude by applying  Lemma \ref{le:totalprogeny}.
\end{proof}

We are now ready to prove Lemma \ref{lem:leaves}. 

\begin{proof}[Proof of Lemma \ref{lem:leaves}]
For integer $k \geq 1$, let $\Lambda_k\subset \dR$ be the set of  $\lambda$ such that there exists a tree with $k$ vertices and $\lambda$ is an eigenvalue of the adjacency matrix of this tree. We note that $\cup_k \Lambda_k$ is the set of totally-real algebraic integers by \cite{SALEZ2015249}.  Obviously, $|\Lambda_k|$ is bounded by $k$ times the number of unlabeled trees with $k$ vertices. In particular, for some $C > 1$, 
\begin{equation}\label{eq:ubLk}
| \Lambda_k | \leq C^k,
\end{equation} see Flajolet and Sedgewick \cite[Section VII.5]{MR2483235}. We define $B_{k,\veps} = \{ x \in \dR : \exists \lambda  \in \Lambda_k,  | \lambda  / \sqrt{d_s} - x | \leq \veps 2^{-k-1} / | \Lambda_k| \}$ and $B = \dR \backslash \cup_{k \geq 1} B_{k,\veps}$. By construction, 
$$
\ell (  B^c ) \leq \sum_{k  \geq 1} | \Lambda_k | \frac{ \veps 2^{-k} }{ | \Lambda_k| }= \veps. 
$$ 
Note that if $|\cT|= k$ and $z \in \dH_B$ then $|g_o (z)| \leq  2^{k+1} |\Lambda_k| /\veps$ (since $H = A / \sqrt d_s$ and the Cauchy-Stieltjes transform of a measure is bounded by one over the distance to the support). Let $g^e(z)$ be the law of $g_o(z)$ given $o \notin \mathcal S$. We deduce from  \eqref{eq:ubLk} and Corollary \ref{cor:totalprogeny} with $n=2$ and $q = \pi_1$ that for $z \in \dH_B$, 
\begin{align*}
\dE[ |g^e (z) |^m] & = \sum_{k = 1}^\infty \dE [ |g^e(z)|^m | |\cT| = k ] \dP (|T| = k ) \\
& \leq \sum_{k = 1}^\infty  \frac{ \sqrt \pi_1 }{\pi_e}  \left(2 \sqrt \pi_1 \right)^k \left( \frac{ 2^{k+1} |\Lambda_k|}{ \veps}  \right)^m \\
& \leq \frac{  2^m \sqrt{\pi_1}  }{\pi_e\veps^m }   \sum_{k = 1}^\infty   \left(   2^{m+1} C^m   \sqrt{\pi_1}  \right)^k  \\
& \leq \frac{2^{2m+2} C^m }{\pi_e\veps^m } \, \pi_1 ,
\end{align*}
where at the last line, we have assumed that $\sqrt \pi_1 \leq 2^{-m-2} C^{-m}$ and used that $\sum_{k \geq 1} a^k \leq 2a$ if $0 \leq a \leq 1/2$. Note that this last condition is satisfied if $\pi_1 \leq c_0^m$ with $1/ c_0 = 2^{4} C^{2}$.  We deduce that for some new constant $C >1$, we have for any integer $m \geq 1$ such that $\pi_1 \leq c_0^m$,
\begin{equation*}\label{eq:momentge}
\dE[ |g^e (z) |^m] \leq \left( \frac{C }{\veps} \right)^m \frac{\pi_1}{\pi_e}.
\end{equation*}
In particular, from \eqref{eq:defV} and H\"older inequality, 
\begin{align*}
\dE  \left[ |v(z)|^m \right] & = \frac{1}{d_s^m} \dE   \left| \sum_{i=1}^{N^e} g^e_i (z)   \right|^m   \leq \frac{1}{d_s^m} \dE \left[(N^e)^{m-1}\sum_{i=1}^{N^e} | g^e_i (z)|^m \right]   \leq  \dE [ (N^e)^m ]\left( \frac{C }{  \veps  d_s } \right)^m \frac{\pi_1}{\pi_e}.
\end{align*}
It remains to estimate $\dE [ (N_e )^m ]$. We write  
$$\dE [(\bar N^e)^m] = \dE [ N_e ^m | N_s \geq 1] \leq \dE [ (\bar N^e) ^m ] / ( 1- \pi_e ) \leq 2 \dE [(\bar N^e)^m],$$ indeed, from Lemma \ref{le:pie}, $\pi_e \leq 2 \pi_1 \leq 2 c_0 \leq 1/2$. We have seen that $\bar N^e \stackrel{d}{=} \sum_{i=1}^N \veps_i$ where $N$ has distribution $P$, independent of $(\veps_i)_{i \geq 1}$ an iid sequence of Bernoulli variables with  $\dP ( \veps_i  = 1) = \pi_e = 1 - \dP ( \veps_i = 0)$.   We find
$$
\dE [(\bar N^e)^m] = \dE \left( \sum_{i=1}^{N} \veps_i   \right)^m \leq \sum_{k=1}^m \dE[ N^k]  \pi^k _e \leq \frac{\dE [ N^m ]}{d^m}  \sum_{k=1}^m (d \pi_e)^k  \leq 2 \frac{\dE [ N^m ]}{d^m}  d \pi_e ,
$$
where, in the first sum, we have used that  $\dE \veps_i^l = \pi_e$ for any $l \geq 1$ and decomposed the sum over the number $k$ of distinct indices in a $m$-tuple. At the next step, we have used that $\dE N^k / d^k$ is non-decreasing in $k$ from Jensen inequality. In the final inequality, we used  Lemma \ref{le:pie} and the assumption $d \pi_e \leq \dE N^2 \pi_e  / d \leq 1/2$.  We thus have proved that 
$$
\dE  \left[ |v(z)|^m \right]  \leq 4 \left( \frac{C }{  \veps  d_s } \right)^m  \frac{\dE [ N^m ]}{d^m}  d \pi_1. 
$$
Finally, from \eqref{eq:varphis} and the convexity of $\varphi'$, we have 
$$
d_s = \partial_1 \varphi_s (1,1) = \varphi'(1 - \pi_e) \geq  \varphi'(1) - \pi_e \varphi''(1) = d - \pi_e \dE [ N(N-1)] \geq d /2.
$$
The claim of the lemma for $v(z)$ follows by adjusting the constant $C$. Replacing $N_\star$ by $N$, the claim on $v_\star(z)$ follows. \end{proof}

\section{From resolvent to spectral measures}
\label{sec:specmeas}

\subsection{Absolutely continuous part of a random measure}

In this final section, we prove the main results in introduction, Theorem \ref{thm:POI}, Theorem \ref{thm:POIcore} and Theorem \ref{thm:Anderson}. This is done by standard tools connecting the Cauchy-Stieltjes transform of a measure and its absolutely continuous part.

We will use the following lemma. Recall the definition of $\dH_B$ in \eqref{eq:defHB}.

\begin{lemma}\label{lem:specmeas1}
Let $\mu_1$ and $\mu_2$ be two random probability measures on $\dR$ with Cauchy-Stieltjes transforms $g_1$ and $g_2$ respectively. Assume that for some deterministic Borel set $B\subset \dR$, $C >0$ and $p > 0$ we have 
$$
\liminf_{\eta \downarrow 0} \int_B \dE [ | g_1 (\lambda + i \eta) - g_2(\lambda + i \eta ) |^p ] \mathrm{d}\lambda \leq C.
$$
Then, if $f_1$ and $f_2$ are the densities of the absolutely continuous parts of $\mu_1$ and $\mu_2$, we have
$$
\dE \int_B | f_1(\lambda) - f_2 (\lambda) |^p \mathrm{d}\lambda \leq \frac{C}{\pi^p}.
$$
\end{lemma} 
\begin{proof} For $\eta >0$ and real $\lambda$, we set $f^\eta_i (\lambda) = \Im g_i(\lambda + i \eta) / \pi$. For $i =1,2$, almost everywhere $\lim_{\eta \to 0} f^\eta_i (\lambda) = f_i (\lambda)$, see for example Simon \cite[Chapter 11]{MR2154153}. We consider the measure on $\dC \times B$: $M = \dP \otimes \ell(\cdot)$. By assumption, from Fubini's Theorem, for any $\eta >0$, 
$$
\liminf_{\eta \downarrow 0} \int |f^\eta_1 - f^\eta_2|^p \mathrm{d}M =  \liminf_{\eta \downarrow 0}   \int_B \dE [ | f^\eta_1(\lambda) - f^\eta_2 (\lambda) |^p ] \mathrm{d}\lambda \leq \frac{C}{\pi^p}.
$$
Moreover, from what precedes, $M$-almost everywhere, $|f^\eta_1 - f^\eta_2|^p$ converges to $|f_1 - f_2|^p$. The conclusion is then a consequence of Fatou's lemma.  \end{proof}
 
The following theorem due to Simon \cite[Theorem 2.1]{simon1995L}  generalizes the case $p=2$ due to Klein \cite[Theorem 4.1]{klein1998extended} (see also (4.46) there). It is a useful criterion for a random probability measure to be purely absolutely continuous. 

\begin{lemma}[\cite{simon1995L,klein1998extended}]\label{lem:specmeas2}
Let $p > 1$ and $\mu$ be a random probability measure on $\dR$ with Cauchy-Stieltjes transform $g$. For any $E >0$, if 
$$
\liminf_{\eta \downarrow 0} \int_{-E}^E \dE [ |g(\lambda + i \eta)|^p ] \mathrm{d} \lambda < \infty
$$
then, with probability one, $\mu$ is absolutely continuous on $(-E,E)$. 
\end{lemma}

 \subsection{Proof of Theorem \ref{thm:Anderson}}

\noindent {\em Step 1: absolute continuity. } Let $H = A / \sqrt {d-1} + \lambda V/\sqrt{d-1}$ and $g_o (z) = \langle \delta_o, (H-z)^{-1} \delta_o\rangle$. By construction $g_o$ is the Cauchy-Stieltjes transform of $\tilde \mu_o  = \mu_o ( \cdot / \sqrt{d-1}  )$. We apply Corollary \ref{cor:main} with $N_\star = d$, $N = d-1$, independent of $v_\star  \stackrel{d}{=} v \stackrel{d}{=}  \lambda V_o / \sqrt{d-1}$. With $\rho = d/(d-1) \geq 1$ and $\alpha$, $\beta$ as in Corollary \ref{cor:main} with $p=2$, we find
 $$
 \alpha(t) = 4 \dE \left[ \IND_{ |v| > t } ( 1+ |v|)^4 \right]  \quad \hbox{ and } \quad \beta(t) = 4 \dE \left[ \IND_{ |v| > t \rho } \left( 1+ \frac{|v|}{\rho} \right)^4  \right].
 $$
Note that $\beta(t)\leq \alpha(t)$. Also, by assumption $\dE |V_o|^4 < \infty$. Since $(a+b)^n \leq 2^{n-1} (|a|^n + |b| ^n)$,
$$
\alpha(t) \leq 2^6 \left( \dP \left( |V_o| > \frac{t \sqrt{d-1} }{\lambda}\right)  + \frac{  \lambda^4}{(d-1)^2} \dE \left[ \IND_{ |V_o| > \frac{t \sqrt{d-1} }{\lambda}} |V_o|^4\right] \right) \leq 2^6 \left( 1+ \frac{1}{t^4} \right) \frac{\lambda^4}{(d-1)^2} \dE  |V_o|^4.
$$
It follows that for each $t,\kappa>0$, there exists $\lambda_1 >0$ such that $\lambda \leq \lambda_1 \sqrt{d}$ implies $\alpha(t),\beta(t) \leq \kappa$. Also $1 \leq \rho \leq 3/2$ (since $d \geq 3$, $d-1 \geq (2/3) d$). Now, we fix $E >0$, we apply this observation to $t = \delta/2$ and $\kappa = \delta/2$ with $\delta$ as in Corollary \ref{cor:main}. We deduce that for $\lambda \leq \lambda_1 \sqrt d$, we have $\dE [|g_o(z)|^2] \leq C$ for some $C$. From Lemma \ref{lem:specmeas2}, we deduce that with probability one, $\mu_o$ is absolutely continuous on $(-E\sqrt{d-1} , E \sqrt {d-1})$.

\noindent {\em Step 2: positive density. } We now prove that $\mu_0$ has positive density on $(-E\sqrt{d-1} , E \sqrt {d-1})$ or equivalently that $\tilde \mu_0$ has positive density on $(-E,E)$. Corollary \ref{cor:main} implies that, for some $C>0$, for all $z \in \dH_E$, 
$$ \dE (\Im g_o (z))^{-2} \leq C .$$
Let $f$ be the density of the absolutely continuous part of $\tilde \mu_o$. Almost-everywhere, we have $\lim_{\eta \to 0} \Im g_o(\lambda + i \eta)  =  f (\lambda)/\pi$. We deduce from Fatou's Lemma that
$$
\dE \int_{-E}^E \frac{1}{f(\lambda)^2 } \mathrm{d}\lambda < \infty.
$$
In particular, with probability one, $\int_{-E}^E  1/ f(\lambda)^2  \mathrm{d}\lambda < \infty$ and $f(\lambda) >0$ almost everywhere.

\noindent {\em Step 3: mass of absolutely continuous part. } It remains to prove that for any $\veps > 0$, there exists $\lambda_0$ such that  $\dE \mu_o^{\ac} (\dR) \geq 1 - \veps$. There exists $E < 2$ such that $\mu_{\star} ( (-E, E) ) \geq 1 - \veps / 2$ uniformly in $d \geq 3$. As explained above, for each $t,\kappa >0$, there exists $\lambda_1 >0$ such that $\lambda \leq \lambda_1 \sqrt{d}$ implies $\alpha(t),\beta(t) \leq \kappa$. By adjusting the value of $t = \kappa$, by Lemma \ref{lem:specmeas1}, we deduce that if $\lambda \leq \lambda_0 (\veps)\sqrt{d}$, we have 
$$
\dE \int_{-E}^E |f(\lambda) - f_{\star} (\lambda) | d \lambda \geq \frac{\veps}{2},
$$ 
where $f$ is the density of the absolutely continuous part of $\tilde \mu_o$ and $f_\star$ the density of $\mu_\star$. 
Therefore, from the triangle inequality, we find
$$
\dE \mu_o^{\ac} (\dR) \geq \dE \int_{-E}^E f(\lambda) d \lambda \geq \int_{-E}^E f_\star(\lambda) d\lambda - \frac{\veps}{2} \geq 1 - \veps,
$$
as requested.
 \qed

  \subsection{Proof of Theorem \ref{thm:POIcore}}
We treat the case of $P_\star = Q_d$ is $\mathrm{Bin}(n,d/n)$ conditioned on being at least $2$. The Poisson case is essentially the same (take $n$ to infinity in the argument below). Let $N_\star$ be a random variable with law $Q_d$ and $N$ with law $P = \widehat{Q}_d$.
The moment generating function of $P_\star$ is 
$$
\varphi_\star(x) = \dE [ x^{N_\star} ] =  \left( \left(1 + \frac d n (x - 1) \right)^n - q_0 - q_1 x \right) / ( 1 - q_0 - q_1),
$$     
where $q_0 = (1 - d/n)^n$ and $q_1 = d (1 - d/n)^{n-1}$ are the probabilities that $\Bin(n,d/n)$ is $0$ and $1$. In particular, the moment generating function of $P$ is 
$$
\varphi(x) = \dE [ x^{N} ]  = \frac{\varphi_\star'(x) }{\varphi_\star' (1) }= \left( \left(1 + \frac d n (x - 1) \right)^{n-1}  - p_0 x \right) / ( 1- p_0),
$$
  where $p_0 = q_1 / d = ( 1- d/n)^{n-1}$. Therefore, $P$ is $\Bin(n-1,d/n)$ conditioned on being at least $1$.

  It is straightforward to check that $q_0,q_1$ and $p_0$ converge to $0$ as $d$ goes to infinity, uniformly in $n \geq d$. More is true: from Bennett's inequality, if $Z = Z_n \stackrel{d}{=} \Bin(n,d/n)$, for any $\lambda > 0$,
  \begin{equation}\label{eq:bennett}
  \dP ( | Z - d | \geq \lambda d)  \leq 2 \exp ( - d h ( \lambda ) ),  
  \end{equation}
with $h(u) = ( 1+ u) \ln (1+ u) - u$. We deduce that for $m \geq 1$,  
$$
\left| \frac{\dE Z^m}{d^m} - 1 \right| \leq  \dE \left| \left( \frac{Z }{ d} \right)^m  - 1 \right|  \leq  m \int_0 ^\infty   d^{m-1} \dP ( | Z/d - 1 | \geq \lambda) \mathrm{d}\lambda  
$$
goes to $0$ as $d$ goes to infinity, uniformly in $n \geq d$. As a consequence, since $\dE Z_{n-1}^m \leq \dE N^m  \leq \dE Z_{n-1}^m / ( 1 -p_0)$ and $\dE Z_{n}^m \leq \dE N_\star^m  \leq \dE Z_{n}^m / ( 1 -q_0 - q_1)$,  for any integer $m \geq 1$, $\dE (N / d) ^m $ and $\dE (N_\star / d)^m$ go to $1$ as $d$ goes to infinity, uniformly in $n \geq d$.

Similarly, from \eqref{eq:bennett}, we write for $0 < \lambda < 1$,
$$
\dE \left(\frac{d}{\max(Z,1) }\right)^m \leq d^m \dP ( Z \leq ( 1- \lambda) d)  + (1- \lambda)^m \leq 2 d^m e^{-d h (\lambda)}  + (1- \lambda)^m. 
$$
We deduce that $\dE (d/ N) ^m $ and $\dE (d/ N_\star )^m$ go to $1$ as $d$ goes to infinity, uniformly in $n \geq d$.

The rest of the proof is the same than the proof of Theorem \ref{thm:Anderson}: we apply Corollary \ref{cor:main} with $p=2$ to $H = A / \sqrt{\dE N}$. The above analysis shows that, for any $\lambda >0$, $\alpha_2(\lambda)$ and $\beta_2(\lambda)$ goes to $0$ as $d$ goes to infinity, uniformly in $n\geq d$. \qed

  \subsection{Proof of Theorem \ref{thm:POI}}
Again, we only treat the binomial case, the Poisson case being essentially the same. Let $\pi_e$ be the extinction probability. Let $N^s_\star$ and $N^s$ be the offspring distribution of the root and $v \ne o$ conditioned on being in $\mathcal S$.  The moment generating functions of $N^s_\star$ and $N^s$ are given by \eqref{eq:defMGFs}. We set $d_s = \dE N^s$. Arguing as in the proof of Theorem \ref{thm:POIcore}, we find that $N^s_\star$ has law $\Bin(n,(1-\pi_e)d / n)$ conditioned on being at least $1$ and that $N^s$ has law $\Bin(n-1,(1-\pi_e)d / n)$ conditioned on being at least $1$. 
By Lemma \ref{le:pie}, $\pi_e$ goes to $0$ as $d$ goes to infinity, uniformly in $n \geq d$. Also, from the proof of Theorem \ref{thm:POIcore}, we deduce that, for any integer $m \geq 1$, 
$$
\dE\left( \frac{ N^s}{d} \right)^m  , \; \dE\left( \frac{d}{ N^s} \right)^m , \; \dE\left( \frac{ N_\star^s}{d} \right)^m  , \; \dE\left( \frac{d}{ N_\star^s} \right)^m  
$$
go to $1$ as $d$ goes to infinity, uniformly in $n \geq d$. Also, if $d \geq 2$, from Bennett's inequality \eqref{eq:bennett},
$$
 \pi_1 =  \dP ( N \leq 1) \leq \dP ( N \leq d/2) \leq 2 \exp ( - d h (1/2) ) = o(1/d).
$$

Now, let $\veps_0 >0$, $B = B(\veps_0)$ be as in Lemma \ref{lem:leaves}. If $v(z)$, $v_\star(z)$ are as in Lemma \ref{lem:RDEleaves}, we deduce from what precedes that, by Lemma \ref{lem:leaves}, $\dE |v(z)|^m$ and $\dE |v_\star (z) |^m $ go to $0$ as $d$ goes to infinity, uniformly in $n \geq d$ and $z \in \dH_B  $.

We set $H  = A / \sqrt{d_s}$ and let $0 < E < 2$. We denote by $\dE^s$ the conditional expectation given $\cT$ infinite.
By Lemma \ref{lem:RDEleaves}, Corollary \ref{cor:main} and the above analysis with $m = 12$ (see \eqref{eq:holder}), we deduce that for any $\kappa >0$, for all $d \geq d_0 (\veps_0,E, \kappa)$ large enough and for all $z \in \dH_E \cap \dH_B $, 
$$
\dE^s |g_o (z) - \Gamma_\star (z) |^2 \leq \kappa/\rho^2 \quad \dE^s (\Im g_o (z))^{-2} \leq C_\star \left(\rho+\frac{1}{\rho}\right)^2,
$$
where $g_o (z) = \langle \delta_o , (H-z)^{-1} \delta_o \rangle $ and $\rho  = \dE N^s_\star / \dE N^s$ can be taken to be approximately close to $1$ for $d$ large enough.

Let $\tilde \mu_o = \mu_o ( \cdot / \sqrt{d_s})$ and let $f(\lambda)$ be the density of the absolutely continuous part of $\tilde \mu_o$. We fix $\kappa >0$ and $d \geq d_0 ( \veps_0,\kappa)$. The argument in step 2 of the proof of Theorem \ref{thm:Anderson} proves that
$$
\dE ^s\int_{B \cap (-E,E)} \frac{1}{f(\lambda)^2 } < \infty.
$$
In particular, conditioned on non-extinction, with probability one, $f(\lambda) >0$ almost-everywhere on $B\cap (-E,E)$. This proves the first claim of the theorem.

For the second claim, we fix $\veps >0$. We have $\dE \mu^{\ac}(\dR) \geq (1- \pi_e) \dE^s \mu^{\ac}(\dR)$ and $\pi_e$ goes to $0$ as $d$ goes to infinity, uniformly in $n \geq 2$. Hence, up to adjusting the value of $\veps$, it suffices to prove that $\dE^s \mu^{\ac}(\dR) \geq 1- \veps$. Moreover, since $\ell (B^c) \leq \veps_0$, there exists $(\veps_0,E)$ such that for all $d$ large enough,
$ \mu_\star ( B \cap (-E,E) ) \geq  1 - \veps/2$. We may then repeat step 3 in the proof of Theorem \ref{thm:Anderson} and conclude that $\dE^s \int_{B \cap (-E,E)} f(\lambda) d\lambda \geq 1 - \veps$ for all $d$ large enough, uniformly in $n$. \qed
 
\bibliographystyle{abbrv}
\bibliography{biblio}

\end{document}